\documentclass[a4paper,11pt,reqno]{amsart}

\usepackage{graphicx}
\usepackage{amsthm}
\usepackage{float}
\usepackage{amsmath,latexsym}
\usepackage{amssymb}
\usepackage{geometry}
\usepackage{amsfonts}
\usepackage{hyperref}
\usepackage[skip=4pt]{subcaption}
\usepackage[ruled]{algorithm2e} 

\def\R{\mathbb{R}}

\def\H{\mathcal{H}}
\def\sign{\mathrm{sign}}
\newcommand{\dsum}{\displaystyle\sum}

\usepackage{pgfplots}

\newtheorem{thm}{Theorem}[section]
\newtheorem{defn}{Definition}[section]
\newtheorem{lem}{Lemma}[section]

\newtheorem{rmk}{Remark}[section]

\allowdisplaybreaks

\let\origmaketitle\maketitle
\def\maketitle{
  \begingroup
  \def\uppercasenonmath##1{} 
  \let\MakeUppercase\relax 
  \origmaketitle
  \endgroup
}

\begin{document}

\title{\large Optimal arrangements of hyperplanes for multiclass classification}

\author[V. Blanco, A. Jap\'on \MakeLowercase{and} J. Puerto]{{\large V\'ictor Blanco$^\dagger$, Alberto Jap\'on$^\ddagger$ and  Justo Puerto$^\ddagger$}\medskip\\
$^\dagger$IEMath-GR, Universidad de Granada\\
$^\ddagger$IMUS, Universidad de Sevilla}

\address{IEMath-GR, Universidad de Granada, SPAIN.}
\email{vblanco@ugr.es}

\address{IMUS, Universidad de Sevilla, SPAIN.}
\email{ajapon1@us.es; puerto@us.es}

\date{}

\maketitle

\begin{abstract}
In this paper, we present a novel approach to construct multiclass classifiers by means of arrangements of hyperplanes. We propose different mixed integer (linear and non linear) programming formulations for the problem  using extensions of widely used measures for misclassifying observations where the \textit{kernel trick} can be adapted to be applicable. Some dimensionality reductions and variable fixing strategies are also developed for these models. An extensive battery of experiments has been run which reveal the powerfulness of our proposal as compared with other previously proposed methodologies.
\end{abstract}

\section{Introduction}

Support Vector Machine (SVM) is a widely-used methodology in supervised binary classification, firstly proposed by Cortes and Vapnik \cite{cortesvapnik95}.
Given a number of observations with their corresponding labels, the SVM technique consists of finding a strip in the feature space so that each class is included in a different semispace maximizing the separation between classes (in a training sample) and minimizing some measure of the misclassification errors. This problem can be cast within the class of convex optimization and its dual enjoys very good properties. Actually, one can project the original data out onto a higher dimensional space where the separation of the classes can be more adequately performed, and still keeping the same computational effort that was required in the original problem. This fact is the so-called \textit{kernel trick}, and very likely this is one of the reasons that has motivated the successful use of this tool  in a wide range of applications \cite{writing,credit,insurance,cancer,cleveland}.

Most of the SVM literature concentrates on binary classification where several extensions are available. One can use different measures for the separation between classes~\cite{BPR18,IkedaMurata05a,IkedaMurata05b}, select important features~\cite{LMR18}, apply regularization strategies~\cite{LMC18}, etc. However, the analysis of SVM-based methods for datasets with more than two classes has been, from our point of view, only partially investigated. The $k$-label ($k>2$) SVM consists of the following. Given a \textit{training sample} of observations $\{x_1, \ldots, x_n\} \subseteq \R^p$ with their  labels $(y_1, \ldots, y_n) \in \{1, \ldots, k\}^n$, the goal is to construct a decision rule able to classify out-of-sample observations learning from the training sample.

The most common techniques applied to supervised multiclass classification are based on natural extensions of the tools valid for the binary case:  Deep Learning \cite{Agarwal_2018}, $k$-Nearest Neighborhoods \cite{knn1,knn} or Na\"ive Bayes \cite{nb}, among others.

In addition, one can also find some techniques for multiclass classification that take advantage of the SVM methods for binary classification. The most popular multiclass SVM-based approaches are One-Versus-All (OVA) and One-Versus-One (OVO). The former, namely OVA, computes, for each class $r \in \{1, \ldots, k\}$, a binary SVM classifier labeling the observations as $1$, if the observation is in the class $r$ and $-1$ otherwise. The process is repeated for all classes ($k$ times), and then each observation is classified into the class whose constructed hyperplane is the furthest from it in the positive halfspace. In the OVO approach, classes are separated with ${k}\choose{2}$ hyperplanes using one hyperplane for each pair of classes, { where the decision rule comes from a voting strategy in which the most represented class among votes becomes the class predicted}. OVA and OVO inherit most of the good properties of binary SVM. In spite of that,  they are not able to correctly classify datasets where separated clouds of observations may belong to the same class (and thus are given the same label) { when a linear kernel is used}. { Another popular method is the directed acyclic graph SVM (DAGSVM) \cite{DAGSVM}. In this technique,  although the decision rule involves the same hyperplanes built with the OVO approach, it is not given by a unique voting strategy but for a sequential number of voting in which the most unlikely class is removed until only one class remains}. { In addition, apart from OVA and OVO, there are other methods based on decomposing the multiclass problem into several binary classification problems. In particular, in \cite{Allwein,Bakiri}, this decomposition is based on the construction of a coding matrix that determines the pairs of classes that will be used to build the separating hyperplanes.} Alternatively, other methods such as Cramer-Singer (CS)~\cite{cs}, Weston-Watkins (WW)~\cite{ww} or Lee-Lin-Wahba (LLW)~\cite{lee2004}, do not address the classification problem sequentially but as a whole considering all the classes within the same optimization model. Obviously, this seems to be the correct approach. { In particular, in WW, $k$ hyperplanes are used to separate the $k$ classes, each hyperplane separating one class from the others, using $k-1$ misclassification errors for each observation. The same separating idea, is applied in CS but reducing the number of misclassification errors for each observation to a unique value. In LLW, a \textit{sum-to-zero constraint} is used to reduce the dimensionality of the problem. We can also find a quadratic extension based on LLW proposed in \cite{Guermeur}. Finally, in \cite{genSVM}, the authors propose a multiclass SVM-based approach,  \textit{GenSVM}, in which the classification boundaries for a problem with $k$ classes are obtained in a $(k-1)$-dimensional space using a simplex encoding.
Some} of these methods have become popular and are implemented in  most software packages in machine learning as \texttt{e1071} \cite{e1071}, \texttt{scikit-learn} \cite{python} { or \texttt{MSVMpack} \cite{MSVMpack} }. Nevertheless, as far as we are concerned, none of the existing multiclass SVM methods  keeps the essence of binary SVM which stems from finding a globally optimal partition of the feature space.

This paper proposes a novel approach to handle multiclass classification extending the paradigm of binary SVM classifiers. In particular, our method  finds a polyhedral partition of the feature space and an assignment of classes to the \emph{cells} of the partition, by maximizing the separation between classes and minimizing two intuitive misclassification errors. Obviously, as in standard SVM, we can also account in different ways the misclassification errors (hinge or ramp-based losses). For bi-class instances, and using a single separating hyperplane, the method coincides with the standard SVM. Nevertheless, even for 2-classes datasets, new alternatives appear if more than one hyperplane to separate the data is permitted. In particular, our approach allows one to generalize the polyhedral conic classifiers presented in \cite{bagirov}.

Apart from justifying the rationale of our method, we also propose different mathematical programming formulations in order to solve the resulting optimization problems. These formulations belong to the family of Mixed Integer (Linear and Non Linear) Programming (MILP and MINLP) problems, in which the nonlinearities come from the representation of the Euclidean distance margin between classes, that can be modeled as a set of second order cone constraints \cite{BPH14}. This type of constraints can be handled nowadays by any of the most popular off-the-shelf optimization solvers (CPLEX, Gurobi, XPress, SCIP, ...).

These models also have a combinatorial nature induced by the correct allocation of labels to cells. Therefore, they require using some binary variables. This approach is not new and recently, a few attempts have been proposed for different classification problems using discrete optimization tools. For instance, in \cite{uney-turkay} the authors construct classification hyperboxes for multiclass classification, in \cite{bennet-demiriz} the authors provide formulations for SVM with unlabelled data (semi-supervised SVM), and in \cite{ghaddar18,LMR18,labbe14}  mixed integer linear programming tools are provided for feature selection in SVM. Handling a large number of binary variables in the models may become an inconvenient when trying to compute classifiers for medium to large size instances. This inconvenience is alleviated with some preprocessing and dimensionality reduction techniques that are also introduced.

In case the data are, by nature, nonlinearly separable, in classical SVM one can apply the so-called kernel trick to project the data out onto a higher dimensional space where the linear separation has a better performance. The key point is that one does not need to know neither the dimension of the final space nor the  specific transformation that is applied to the data: the resulting mathematical programming problem is in the same space as the original one. Here, we show that the kernel trick can be extended to our framework and therefore, it also allows us to find nonlinear classifiers with this methodology.

To asses the validity of our method we have performed a battery of computational tests on two different families of data. We have tested our method against some well-known multiclass SVM classifiers (OVO, CS, WW and LLW) on 6 databases from the UCI repository. Moreover, we also report results on synthetic datasets specially tailored to capture the difficulty of multiclass supervised classification. In all cases, our methods give results similar or superior to those provided for the other methods. In particular, for the synthetic data instances the improvement in accuracy on the test samples are remarkable (see Table \ref{table4}).

The rest of the paper is organized as follows. In sections \ref{sec:1} and \ref{sec:formulation} we describe and set up the elements of the problem to be considered. Afterward,  we introduce a MINLP formulation for our model. Alternatively,  we also present a linear version, which is obtained whenever we measure the margins with the $\ell_1$-norm. A discussion on the extension, with very few modifications, of the previous models to the Ramp Loss versions is included as well. In Subsection \ref{subsec:kernel} we show how an analogous to the kernel trick can be extended to be applied in this model. Section \ref{sec:heuristic} describes some heuristic strategies, preprocessing and dimensionality reductions to obtain good quality initial solutions of the MINLP. Finally, in section \ref{sec:experiments} we report our computational results on different real and synthetic datasets, and compare our method with the most classical ones for multiclass SVM.

\section{Multiclass Support Vector Machines}\label{sec:1}

In this section, we introduce the problem under study and set the notation used through this paper.

Given a training sample $\{(x_1,y_1), \ldots, (x_n,y_n)\} \subseteq \R^p \times \{1, \ldots, k\}$ the goal of supervised classification is to find a decision rule to assign labels ($y$) to data ($x$), in order to be applied to out-of-sample data. We assume that a given number, $m$, of hyperplanes in $\R^{p}$ have to be built to obtain a subdivision of this space into full dimension polyhedral regions that we shall denote as \textit{cells}. (Here, we would like to mention that the term \textit{cell} stands for a nonempty intersection of the semispaces induced by the hyperplanes in the considered family). Let us denote by $\H_1, \ldots, \H_m$ the hyperplanes to be found, which are in the form $\H_r=\{z \in \R^p: \omega_r^t z + \omega_{r0} =0\}$ for $r=1, \ldots, m$ { (here $v^t$ stands for the transpose operator applied to the vector $v\in \R^p$)}.  Each cell induced with such an arrangement of hyperplanes will be then assigned to a label in $\{1, \ldots, k\}$. In Figure \ref{fig0} we illustrate a subdivision of $\mathbb{R}^2$ induced by 2 hyperplanes and the labels assigned to each cell. In the left figure, we represent the observations, highlighting the classes with different symbols (stars, circles and squares). In the right figure, two hyperplanes which induce 4 cells are constructed to separate the three classes. Each cell is assigned to a class (north $\rightarrow$ circles,  south $\rightarrow$ stars, east $\rightarrow$ stars and west $\rightarrow$ squares). In this example the subdivision in cells  and the assignment of labels reaches a perfect classification on  the given observations.

\usetikzlibrary{shapes.geometric}
\tikzset{
    WNode/.style={circle, fill=black,draw,minimum size=4pt,inner sep=0pt,},
    BNode/.style={rectangle,fill=black,thick,draw,inner sep=0pt,minimum size=4pt},
    YNode/.style={star,fill=black,thick,draw,inner sep=0pt,minimum size=4pt},
    WNode1/.style={circle,draw,black,minimum size=4pt,inner sep=1pt,},
  }
  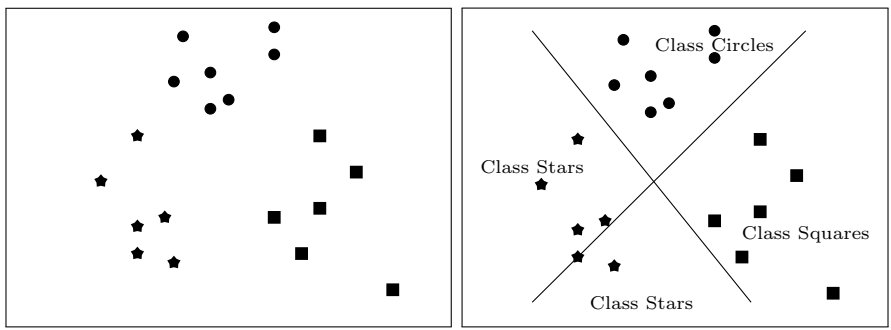
\begin{figure}[h]
\begin{center}  \fbox{\begin{tikzpicture}[
        scale=1.2,
        important line/.style={thick}, dashed line/.style={dashed, thin},
        every node/.style={color=black},
    ]

    \foreach \Point in {(.9,2.4), (1.3,2.5), (1.3,2.1), (2,3), (1,2.9), (1.5, 2.2), (2,2.7)}{
      \draw \Point node[WNode]{};
    }

    \foreach \Point in {(2.9,1.4), (2.3,.5), (3.3,.1), (2,0.9), (2.5,1), (2.5, 1.8)}{
      \draw \Point node[BNode]{};
    }

        \foreach \Point in {(0.9,0.4), (0.5,.5), (0.1,1.3), (0.8,0.9), (0.5,1.8), (0.5, 0.8)}{
      \draw \Point node[YNode]{};
    }

    \draw (3.5,2.9) node {\color{white} \phantom{XX}};
    \draw (-0.5,0) node {\color{white} \phantom{XX}};
  \end{tikzpicture}}~~\fbox{\begin{tikzpicture}[
        scale=1.2,
        important line/.style={thick}, dashed line/.style={dashed, thin},
        every node/.style={color=black},
    ]

    \foreach \Point in {(.9,2.4), (1.3,2.5), (1.3,2.1), (2,3), (1,2.9), (1.5, 2.2), (2,2.7)}{
      \draw \Point node[WNode]{};
    }

    \foreach \Point in {(2.9,1.4), (2.3,.5), (3.3,.1), (2,0.9), (2.5,1), (2.5, 1.8)}{
      \draw \Point node[BNode]{};
    }

        \foreach \Point in {(0.9,0.4), (0.5,.5), (0.1,1.3), (0.8,0.9), (0.5,1.8), (0.5, 0.8)}{
      \draw \Point node[YNode]{};
    }

    \draw (0,1.5) node  {\tiny Class Stars};
    \draw (1.2,0) node   {\tiny Class Stars};

   \draw (3,0.75) node   {\tiny Class Squares};

   \draw (2,3.05) node   {\tiny };
   \draw (2,2.85) node   {\tiny Class Circles};
    \draw (0,0)--(3,3);

        \draw (0,3)--(2.4,0);
  \end{tikzpicture}}
\end{center}
\caption{Illustration of a subdivision induced by 2 hyperplanes in $\mathbb{R}^2$.\label{fig0}}
\end{figure}

From the above, we would like to construct an arrangement of $m$ hyperplanes, $\mathbb{H}=\{\H_1, \ldots, \H_m\}$,  determined by $\omega_1, \ldots, \omega_m \in \R^{p+1}$ (the first component of each vector accounts for the intercept) and a decision rule that assigns a single label to each one of the cells in the subdivision of the space induced by such an arrangement. We would like to point out that each cell in the subdivision can be univocally  identified with a $\{-1,+1\}$-vector in $\R^m$: the $\ell$-component of that vector represents the side (positive or negative) with respect to the hyperplane $\H_\ell$ where that cell lies in.

\begin{defn}[Suitable Assignment]
Given a subdivision $\mathcal{C}$ of $\mathbb{R}^p$ into cells induced by the arrangement of hyperplanes $\mathbb{H}=\{\H_1, \ldots, \H_m\}$ in $\R^p$, a function  $g: \{-1,1\}^m \rightarrow \{1, \ldots, k\}$ is said a \emph{suitable assignment}, if  $g$ univocally maps cells (equivalently, sign-patterns) to labels in $\{1, \ldots, k\}$.
\end{defn}

Observe that a suitable assignment, $g$, allows us to classify any observation $x \in \R^p$ within the set of classes $\{1, \ldots, k\}$, as follows:
\begin{enumerate}
\item Identify $x$ with a sign-pattern:  $\mathrm{s}(x)=(\mathrm{s}_1(x), \ldots, \mathrm{s}_m(x)) \in \{-1,+1\}^m$, where $\mathrm{s}_r(x) = \sign(\omega_r^t x +\omega_{r0})$ for $r=1,\ldots, m$.
\item Apply the function $g$ to the sign-patterns: $\widehat{y}(x) = g(\mathrm{s}(x)) \in \{1, \ldots, k\}$, is the \textit{predicted} label of $x$.
\end{enumerate}

The quality of the decision rule is based, on comparing predictions and actual labels on a training sample, but also on maximally separating the classes in order to find good predictions and avoid undesired overfitting.

In binary classification datasets, SVM is a particular case of our approach if $m=1$, i.e., a single hyperplane to subdivide the feature space is used. In such a case, signs are in $\{-1,1\}$ and classes in $\{1,2\}$, so whenever there are observations in both classes, the assignment is one-to-one. However, even for biclass instances, if more than one hyperplane is used, one may find better classifiers { (we illustrate this behavior with the dataset \texttt{2C4N} of our computational experiments in Table \ref{table4}).}
In Figure \ref{fig:a}, left-and-right, we draw the same dataset of labeled (red and blue) observations and the result of applying a standard SVM (left) and our method with $2$ hyperplanes. In that picture one may see that not only the misclassification errors are smaller with two hyperplanes, as expected, but also the separation between classes is larger, improving the predictive power of the classifier.

\begin{figure}[h]
\begin{center}
\includegraphics[scale=0.26]{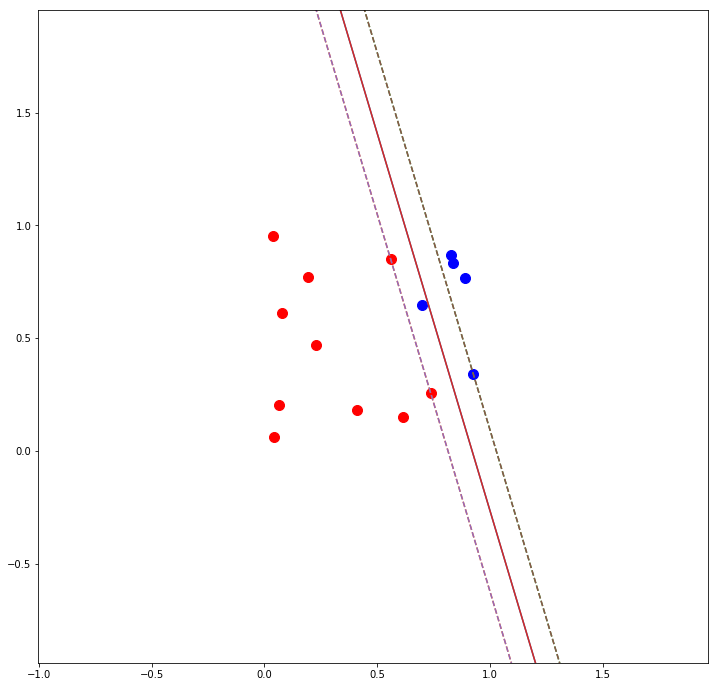}~\includegraphics[scale=0.26]{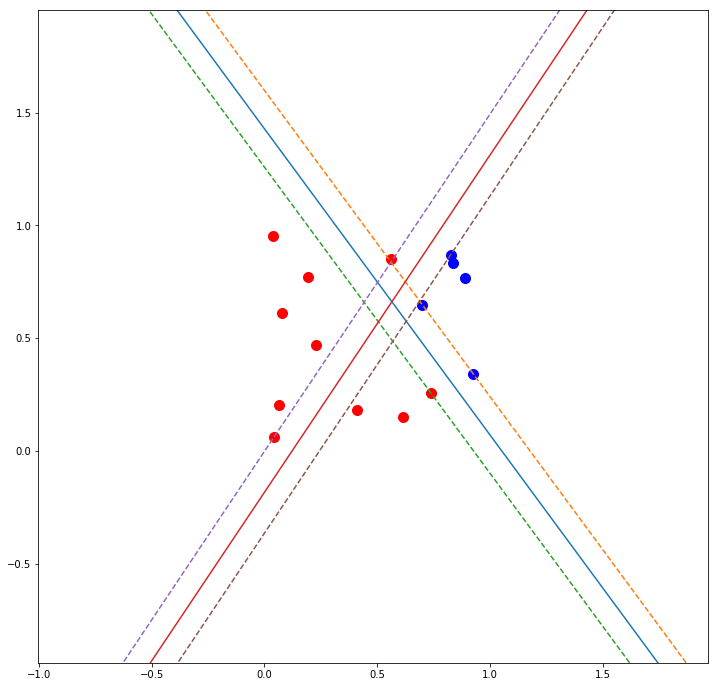}
\caption{Standard SVM (left) and our approach with $2$ hyperplanes (right).\label{fig:a}}
\end{center}
\end{figure}
The rationale of our approach is particularly adequate for datasets in which there are several separated ``clouds'' of observations that belong to the same class. In Figure \ref{fig:JJW2}, we show two different instances in which, again, the colors indicate the class of the observations. The classes in both instances cannot be appropriately separated using any of the available linear SVM-based methods in the literature { since they are based on subdividing the space on class-connected regions}. However,  we are able to perfectly separate the classes using 5 hyperplanes.
\begin{figure}[h]
\begin{center}
\includegraphics[scale=0.28]{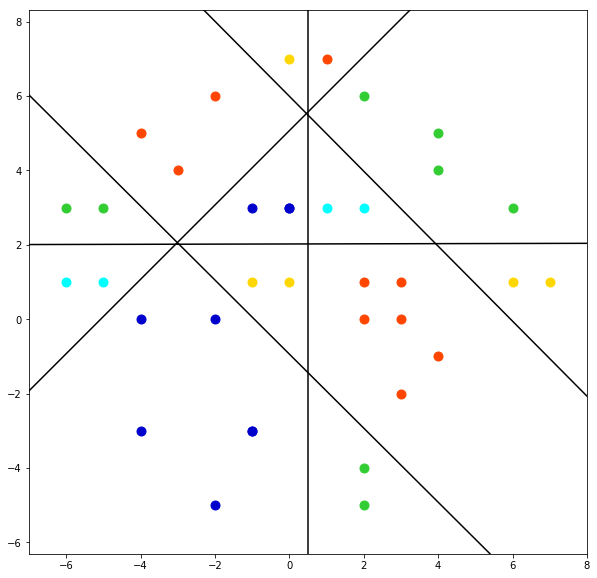}~\includegraphics[scale=0.28]{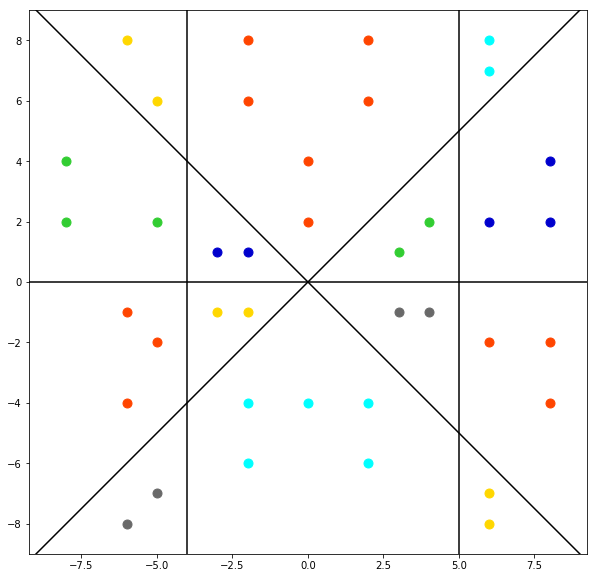}
\caption{A $5$-classes instance classified with our approach using $5$ hyperplanes (left) and a $6$-classes instance classified with our approach using $5$ hyperplanes (right).\label{fig:JJW2}}
\end{center}
\end{figure}

In Figure \ref{fig:AAL1} we compare our approach and the One-versus-One (OVO) approach in an instance with $24$ observations. In the left figure we show the result of separating the classes with four hyperplanes, reaching a perfect classification on the training sample. In the right figure we show the best linear OVO classifier, in which only $66\%$ of the data were correctly classified. We would like also to highlight that, although nonlinear SVM-approaches may separate the data more conveniently, our approach may help to avoid using kernels and ease the interpretation of the results.

\begin{figure}[h]
\begin{center}
\includegraphics[scale=0.1]{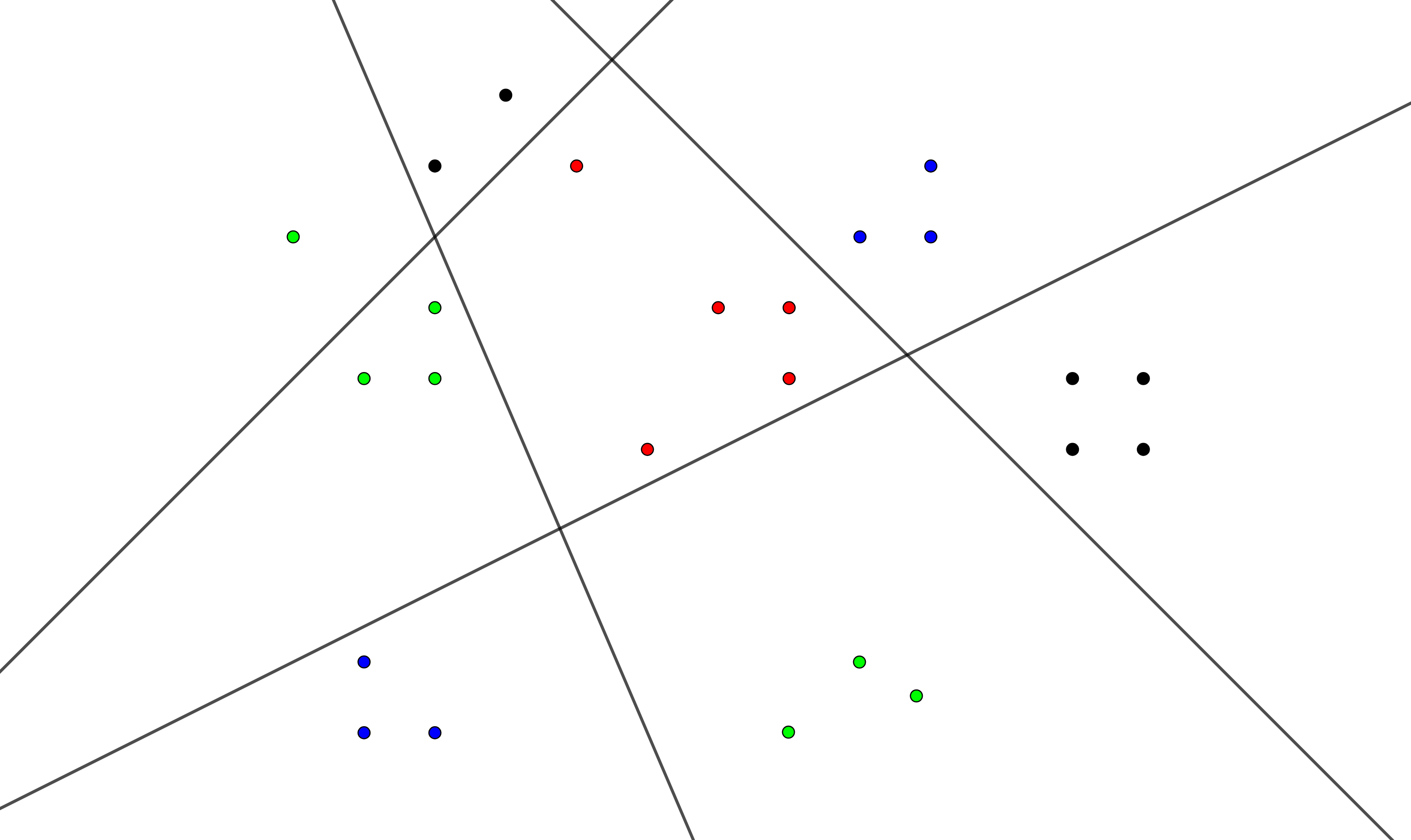}~\includegraphics[scale=0.1]{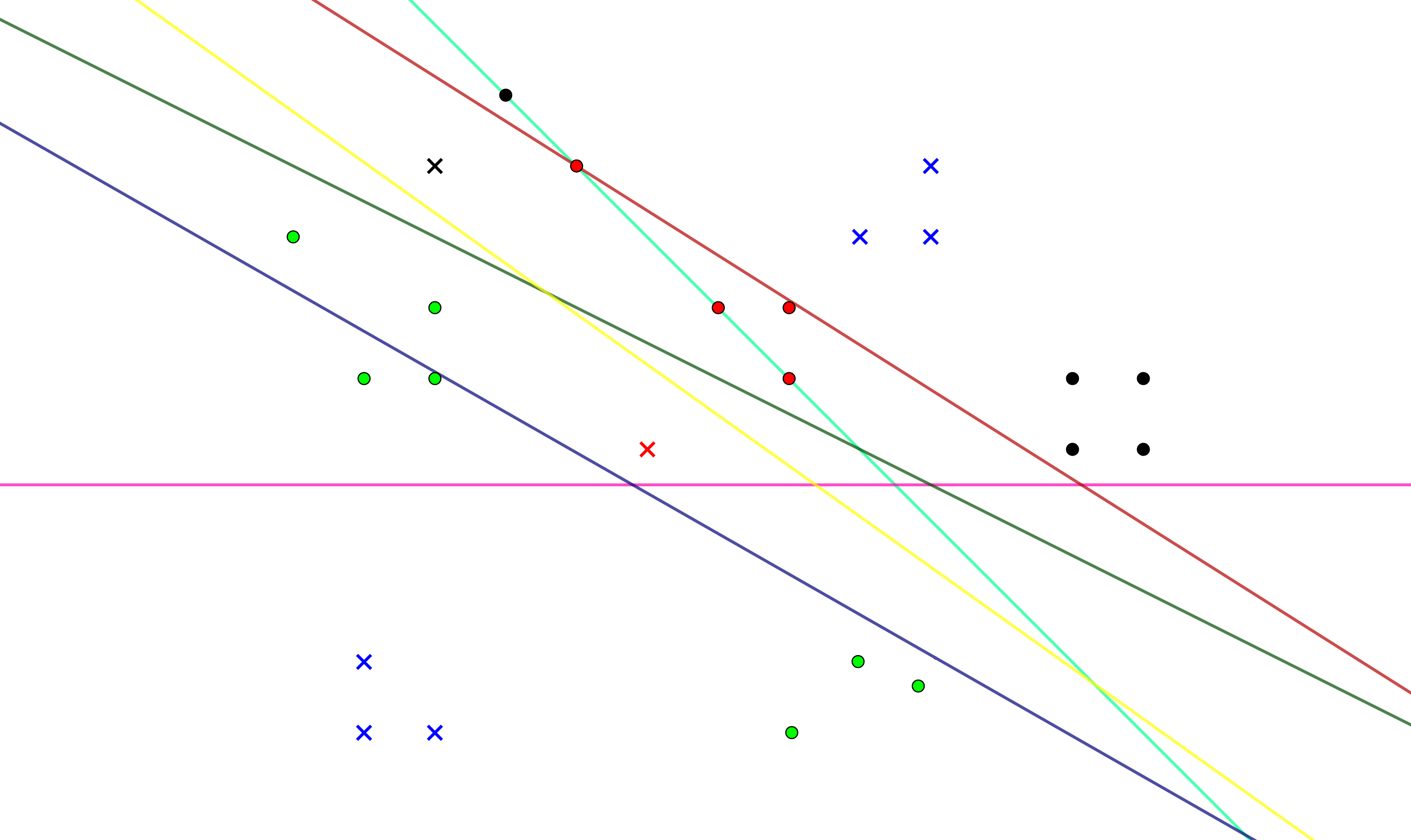}
\caption{A $4$-classes instance classified with our approach using $4$ hypeplanes (left) and the same intance classified using the OVO SVM approach (right).\label{fig:AAL1}}
\end{center}
\end{figure}

Different alternatives could be admissible to justify the rationale of the multiclass classifiers in our framework. To simplify the presentation, we will concentrate on two different models which share the same paradigm but differ in the way they account for misclassification errors. Recall that in SVM-based methods, two criteria are \textit{simultaneously} optimized when constructing a classifier. On the one hand,  a measure of the quality of the decision rule on out-of-sample observations, based on finding a maximum separation between classes; and on the other hand a measure of the misclassification errors for the training set of observations. Both criteria are adequately weighted in order to find a good compromise between the two goals.

In what follows we describe how similar measures can be defined in our multiclass classification framework and the way we account them for.

\subsection{Separation between classes \label{ssec:sbc}}

Separation between classes will be measured as it is usual in SVM-based methods. Let $(\omega_1;\omega_{10}), \ldots, (\omega_m; \omega_{m0}) \in \R^p \times \R$ be the coefficients and intercepts of a set of hyperplanes. The distance { induced by a norm $\|\cdot\|$} between the shifted hyperplanes $\H_r^+=\{z\in \R^p: \omega_{r}^t z + \omega_{r0}=1\}$ and $\H_r^-=\{z\in \R^p: \omega_{r}^t z + \omega_{r0}=-1\}$ is given by $\frac{2}{\|\omega_{r}\|^*}$, where $\|\cdot\|$ is a given norm in $\R^p$ { and $\|\cdot\|^*$ is its dual norm }(see \cite{mangasarian}). { Unless explicitly mentioned, we will consider that $\|\cdot\|$ is the Euclidean norm which dual is also the Euclidean norm}. 

Hence, in order to find globally optimal hyperplanes with maximum separation, we maximize the minimum separation between classes, that is $\min \left\{\frac{2}{\|\omega_{1}\|}, \ldots, \frac{2}{\|\omega_{m}\|}\right\}$. This measure will conveniently keep the minimum separation between classes as largest as possible. Observe that finding the maximum min-separation is equivalent to minimize $\max \{\frac{1}{2}\|\omega_{1}\|^2, \ldots, \frac{1}{2}\|\omega_{m}\|^2\}$.  For a given arrangement of hyperplanes, $\mathbb{H}=\{\H_1, \ldots, \H_m\}$, we will denote by $h_H(\H_1, \ldots, \H_m) = \max \{\frac{1}{2}\|\omega_{1}\|^2, \ldots, \frac{1}{2}\|\omega_{m}\|^2\}$.

We note in passing that different criteria could have been used to model the separation between classes. For instance, one may consider to maximize the summation of all separations namely $\sum_{r=1}^m \frac{2}{\|\omega_r\|}$. However, although mathematically possible, this approach does not capture the original concept in classical SVM and we have left it to be developed by the interested reader.

\subsection{Misclassification errors}

The performance of a classifier on the training set is usually measured with some function of the misclassification errors. Classical SVMs with hinge-loss errors use, for non well-classified observations, a penalty proportional to the distance to the side in which they would have been well-classified. Then the overall sum of these errors is minimized. We extend the notion of hinge-loss errors to the multiclass setting as follows.

Let  $\mathbb{H}=\{\H_1, \ldots, \H_m\}$ be an arrangement of hyperplanes and $(x, y)$ a pair observation ($x$), label ($y$), with ${\rm s}(x) = (\mathrm{s}_1(x),  \ldots, \mathrm{s}_m(x))$ being the sign-pattern of $x$ with respect to the hyperplanes in $\mathbb{H}$. Let $g: \{-1,1\}^m \rightarrow \{1, \ldots, k\}$ be a suitable assignment. We denote by ${\rm t}(x) = ({\rm t}_1(x), \ldots, {\rm t}_m(x))$ the signs of the closest cell to $x$ whose class by $g$ is $y$. We will say that $(x, y)$ is \textit{wrong-classified} with respect to $\H_r$ if ${\rm s}_r(x) \neq {\rm t}_r(x)$, otherwise it is said that $(x,y)$ is \emph{well-classified}.

In what follows we describe the different error measures (misclassification errors due to different causes) that will be considered for $x$ in order to construct an optimal decision rule.

\begin{defn}[Multiclass In-Margin Hinge-Loss]
The multiclass in-margin hinge-loss for $(x,y)$ with respect to the hyperplane $\H_r$ is given as:
$$
h_{I}\Big(x,y,\H_r\Big) =  \left\{\begin{array}{cl}\max\{0, 1- \mathrm{s}_r(x) \cdot (\omega_{r}^t x + \omega_{r0})\} &  \mbox{if $x$ is well classified through  $\H_r$,}\\
0  & \mbox{otherwise.}
\end{array}\right.
$$
\end{defn}

Observe that $h_{I}$ models the error due to observations that although adequately classified with respect to $\H_r$, belong to the margin between the shifted hyperplanes $\H_r^+$ and $\H_r^-$. These errors will be zero if the observation is wrong-classified, or if it is well-classified and does not  belong to the margin induced by the $r$-th hyperplane.

\begin{defn}[Multiclass Out-Margin Hinge-Loss]
The multiclass out-margin hinge-loss for $(\bar x, \bar y)$ with respect to the hyperplane $\H_r$ is given as:
$$
h_{O}\Big((x,y,\H_r\Big) = \left\{\begin{array}{cl}
1 - \mathrm{t}_r(x) \cdot (\omega_r^t x + \omega_{r0})  &  \mbox{if $x$ is not well classified through  $\H_r$,}\\
0 & \mbox{otherwise.}
\end{array}\right.
$$
\end{defn}

$h_{O}$ measures, for wrong-classified observations, how far is from being well-classified.  This error is zero whenever an observation is well-classified. Note that if an observation, besides being wrong-classified, belongs to the margin between $\H_r^+$ and $\H_r^-$, then only $h_O$ should be accounted for. In Figure \ref{fig2} we illustrate the differences between the two types of losses.

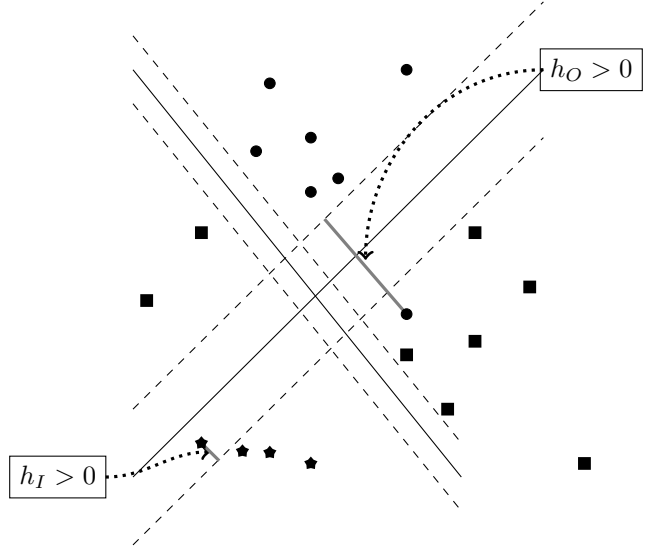
\begin{figure}[h]
\begin{center}
\begin{tikzpicture}[
        scale=1.8,
        important line/.style={thick}, dashed line/.style={dashed, thin},
        every node/.style={color=black},
    ]

    \draw (0,0)--(3,3);
    \draw (0,3)--(2.4,0);

    \draw[dashed] (0,0.5)--(3,3.5);
    \draw[dashed] (0,-0.5)--(3,2.5);

        \draw[dashed] (0,3.25)--(2.4,0.25);
    \draw[dashed] (0,2.75)--(2.4,-0.25);

        \draw [very thick,->, dotted] (-0.2,0) to [out=0,in=180] (0.565,0.185);
        \node[draw] at (-0.55,0) {\small $h_I>0$};

        \draw[very thick,gray] (0.5,0.25) -- (0.63,0.12);

           \draw [very thick,->, dotted] (3,3) to [out=180,in=90] (1.7,1.6);
        \node[draw] at (3.35,3) {\small $h_O>0$};

        \draw[very thick,gray] (2,1.2) -- (1.4,1.9);

            \foreach \Point in {(.9,2.4), (1.3,2.5), (1.3,2.1), (2,3), (1,2.9), (1.5, 2.2), (2,1.2)}{
      \draw \Point node[WNode]{};
    }

    \foreach \Point in {(2.9,1.4), (2.3,.5), (3.3,.1), (2,0.9), (2.5,1), (2.5, 1.8),(0.5,1.8), (0.1,1.3)}{
      \draw \Point node[BNode]{};
    }

        \foreach \Point in {(1.3,0.1), (0.5,.25), (0.8,0.19),  (1, 0.18)}{
     \draw \Point node[YNode]{};
    }
  \end{tikzpicture}
\end{center}
\caption{Illustration of the error measures considered in our approach.\label{fig2}}
\end{figure}

\section{Mixed Integer Non Linear Programming Formulations}\label{sec:formulation}

In this section we describe the two mathematical optimization models that we propose for the multiclass classification problem.  Using the notation introduced in previous sections, the problem can be mathematically stated as follows:
\begin{align}
\min &\; h_H(\H_1, \ldots, \H_m) + C_1 \dsum_{i=1}^n \dsum_{r=1}^m h_{I}\Big(x_i,y_i,\H_r\Big) +  C_2 \dsum_{i=1}^n \dsum_{r=1}^m h_{O}\Big(x_i,y_i,\H_r\Big)\label{m0}\\
\mbox{s.t. } & \mathcal{H}_r \mbox{ is a hyperplane in $\R^p$, for } r=1, \ldots, m.\nonumber
\end{align}

$C_1$ and $C_2$ are parameters which model the \emph{cost} of misclassified and strip-related errors. Usually these constants will be considered equal, nevertheless, in practice analyzing different values for them might lead to better results on predictions. A case of interest results considering $C_2=mC_1$, { i.e., the unitary cost of misclassification errors caused by out-margin observations is $m$ times the unitary cost caused by in-margin observations, giving a larger penalty to wrongly classified observations, avoiding the calibration of a larger number of parameters.}

Observe that the problem above consists of finding the arrangement of hyperplanes minimizing a combination of the three quality measures described in the previous section: 1) the maximum margin between classes, 2) the overall sums of the in-margin errors and 3) the out-margin misclassification errors. In what follows, we describe how the above problem can be re-written as a mixed integer non linear programming problem by means of adequate decision variables and constraints. Furthermore, the proposed model will consist of a set of continuous and binary variables, a linear objective function, and a set of linear and second order cone constraints. It will allow us to push the model to a commercial solver in order to easily solve, at least, small to medium instances.

First, we describe the variables and constraints needed to model the first term in the objective function. We consider the continuous variables $\omega_r \in \R^p$ and $\omega_{r0} \in \R$ to represent the coefficients and intercept of hyperplane $\H_r$, for $r=1, \ldots, m$. Since there is no distinction between hyperplanes, we can assume, without loss of generality that they are non-decreasingly sorted with respect to the norms of their coefficients, i.e., $\|\omega_1\|\geq \|\omega_2\| \geq \cdots \geq \|\omega_m\|$. Then, it is straightforward to see that the term $h_H(\H_1, \ldots, \H_m)$ can be replaced in the objective function by $\frac{1}{2}\|\omega_1\|^2$, once the following set of constraints is included in the model:
\begin{align}
\frac{1}{2}\|\omega_{r-1}\|^2 \geq \frac{1}{2}\|\omega_r\|^2, \forall r=2, \ldots, m.\label{q1}
\end{align}

As already applied in multivariate linear regression~\cite{BPS18} or binary SVM~\cite{BPR18}, other norms can also be used to measure the margin.

For the second term, the in-margin misclassification error,  $h_{I}\Big(x_i,y_i,\H_r\Big)$, corresponding to the observation $(x_i,y_i)$ will be identified with the continuous variable $e_{ir}\geq 0$, for $i=1, \ldots, n$, $r=1, \ldots, m$. Observe that to properly determine each of these errors, one has to determine whether the observation $x_i$ is well-classified or not with respect to the $r$th hyperplane.  In order to do that we need to introduce some binary variables. First, we consider the following two sets of binary variables:
$$
t_{ir}= \left\{\begin{array}{cl} 1 & \mbox{if $\omega_r^t x_i+\omega_{r0} \geq 0$,}\\
0 & \mbox{otherwise.}\end{array}\right. \quad \mbox{ and } \quad z_{is}= \left\{\begin{array}{cl} 1 & \mbox{if $i$ is assigned to class $s$,}\\
0 & \mbox{otherwise.}\end{array}\right.
$$
for $i=1, \ldots, n$, $r=1, \ldots, m$, $s=1, \ldots, k$.  The $t$-variables model the sign-pattern of the observations, while the $z$-variables give the allocation profile of observations to classes. As mentioned above, the classification rule is based on assigning sign-patterns to classes.

The adequate definition of the $t$-variables is assured with the following constraints:
\begin{align}
&\omega_{r}^tx_i + w_{r_0} \geq - T (1-t_{ir}), &\forall i\in N, r \in M\label{t+}\\
& \omega_{r}^tx_i + w_{r_0}  \leq T t_{ir} &\forall i\in N, r \in M\label{t-}
\end{align}
where $T$ is a big enough constant. Observe that $T$ can be accurately estimated based on the data set under consideration.

The following constraints assure the adequate relationships between the variables:
\begin{align}
&\dsum_{s=1}^k z_{is}=1, &\forall i\in N,\label{q2}\\
& \|z_i - z_j\|_1 \leq 2 \|t_i-t_j\|_1, &\forall i, j\in N,\label{q3}
\end{align}
Observe that \eqref{q2} enforce that a single class is assigned to each observation while \eqref{q3} assure that the assignments of two observations must coincide if their sign-patterns are the same. Additionally, the set of $z$-variables  determines whether an observation is well-classified. Indeed, let $\delta_{i}\in \{0,1\}^k$ be defined as $\delta_{is}=1$ if $y_i=s$ and $0$ otherwise. (Observe that $\delta_i$ is the binary encoding of the class of the $i$th observation.) Then, $\xi_i = \frac{1}{2}  \|z_i-\delta_{i}\|_1\in \{0,1\}$ assumes the value zero if and only if the observation $i$ is well-classified, { i.e.,
$$
\xi_i= \left\{\begin{array}{cl} 1 & \mbox{if $i$ is well-classified,}\\
0 & \mbox{otherwise.}\end{array}\right.
$$}

Now, we will model whether the $i$th observation is well-classified or not, with respect to the $r$th hyperplane. Observe that the measure of how far is a wrong-classified observation from being well-classified, needs a further analysis. One may has a wrong-classified observation and several training observations in its same class. We assume that the error for this observation is the misclassification error with respect to the closest cell for which there are well-classified observations in its class. Thus, we need to model the decision on the well-classified \emph{representative} observation for a wrong-classified observation. In Figure \ref{fig:error}, we illustrate this type of misclassification errors. The observation $x_i$ is wrong-classified but the misclassification error of $x_i$, in case $x_j$ is chosen as its representative (well-classified) observation, is $0$ with respect to hyperplane $\mathcal{H}_1$ (note that both $x_i$ and $x_j$ are in the same side of $\mathcal{H}_1$), whereas the misclassification error with respect to $\mathcal{H}_2$ is $h$. Observe $h$ is the distance between $x_i$ and the shifted hyperplane defining the halfspace where $x_j$ lies in. We consider the following set of binary variables:
$$
h_{ij} = \left\{\begin{array}{cl}
1 & \mbox{if $x_j$, which is well classified and verifies $y_j=y_i$, is the representative}\\
 &\mbox{ of $x_i$ in its closest cell  through hyperplanes,}\\
0 & \mbox{otherwise}\end{array}\right.
$$
These variables require to impose the following constraints:
\begin{align}
& \dsum_{j \in N:\atop y_i=y_j} h_{ij}=1, &\forall i\in N,\label{q6}\\
& \xi_j+h_{ij} \leq 1 &\forall i, j\in N (y_i=y_j),\label{q7}\\
& h_{ii} = 1- \xi_i  &\forall i\in N,\label{q7a}
\end{align}
The first set of constraints, \eqref{q6}, impose a single assignment between observations belonging to the same class. Constraints \eqref{q7} avoid choosing wrong-classified representative observations. The set of constraints \eqref{q7a} enforces well-classified observations to be represented by themselves.

\begin{figure}[h]
\begin{center}
\begin{tikzpicture}[
        scale=1.5,
        important line/.style={thick}, dashed line/.style={dashed, thin},
        every node/.style={color=black},
    ]
        \draw[very thick,gray] (2,1.2) -- node[left] {$h$} (1.58,0.78);
    \foreach \Point in {(.9,2.4), (1.3,2.5), (1.3,2.1), (2,3), (1,2.9), (1.5, 2.2), (2,1.2)}{
      \draw \Point node[WNode]{};
    }

    \foreach \Point in {(2.9,1.4), (2.3,.5), (3.3,.1), (2,0.9), (2.5,1), (2.5, 1.8),(0.5,1.8), (0.1,1.3)}{
      \draw \Point node[BNode]{};
    }

        \foreach \Point in {(1.3,-0.1), (1.7,.2), (0.8,0.01)}{
     \draw \Point node[WNode]{};
    }

            \foreach \Point in { (1, 0.18)}{
     \draw \Point node[WNode1]{};
    }


    \draw[dashed, very thin] (0,0.5)--(3,3.5);
    \draw[dashed, very thin] (0,-0.5)--(3,2.5);
\node[right] at (0,0.2){$\mathcal{H}_1$};
\node[right] at (0,2.7){$\mathcal{H}_2$};
\node[right] at (1,0.18){$x_j$};
\node[right] at (2,1.2){$x_i$};
        \draw[dashed, very thin] (0,3.25)--(2.4,0.25);
    \draw[dashed, very thin] (0,2.75)--(2.4,-0.25);


  \end{tikzpicture}
\end{center}
\caption{Illustration of the wrong-classification errors.\label{fig:error}}
\end{figure}
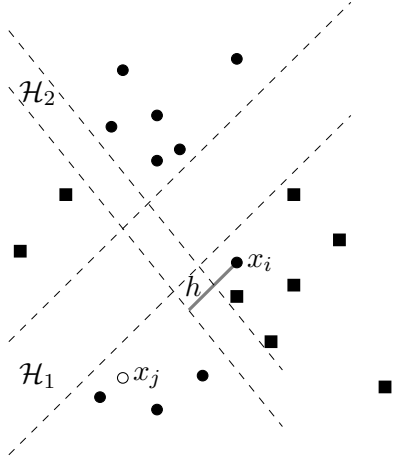

With these variables, we can model the in-margin errors by means of the following constraints:
\begin{align}
&\omega_{r}^t x_i + \omega_{r0} \geq 1-e_{ir}- T\;(3-t_{ir}-t_{jr}-h_{ij}), &\forall r \in M,\label{q4}\\
& \omega_{r}^t x_i +\omega_{r0} \leq -1 +e_{ir}+ T\; (1+t_{ir}+t_{jr}-h_{ij}),& \forall r\in M,\label{q5}
\end{align}
These constraints model, by using the sign-patterns given by $t$, that, $e_{ir} = \max\{0,\min \{1,1- \mathrm{s}_r(x)(\omega_{r}^t x_i + \omega_{r0}\}\}$. Note that the constraints are active if either $t_{ir}=t_{jr}=h_{ij}=1$, i.e., if the well-classified observation $x_j$ is the representative observation for $x_i$ and both are in the positive side of the $r$th-hyperplane; or  $t_{ir}=t_{jr}=0$ and $h_{ij}=1$, i.e.,  if the well-classified observation $x_j$ is the representative observation for $x_i$ and both are in the negative side of the $r$th-hyperplane.  Thus, constraints \eqref{q4} and \eqref{q5} adequately model the in-margin errors for all observations . Furthermore, because of \eqref{t+} and \eqref{t-}, and those described above, the variables $e_{ir}$ always take values smaller than or equal to $1$.

Finally, the third addend, the out-margin errors, will be modeled through the continuous variables $d_{ir}\geq 0$, for $i=1, \ldots, n$, $r=1, \ldots, m$. With the set of variables described above, the out-margin misclassification errors can be adequately modeled through the following constraints:
\begin{align}
& d_{ir} \geq 1- \omega^t_r x_i -\omega_{r0} - T\,(2+t_{ir}-t_{jr}-h_{ij}),& \forall i,j\in N (y_i=y_j), r\in M,\label{q8}\\
& d_{ir} \geq 1+ \omega^t_r x_i + \omega_{r0} - T\,(2-t_{ir}+t_{jr}-h_{ij}),& \forall i,j\in N (y_i=y_j), r\in M,\label{q9}
\end{align}
Constraints \eqref{q8} are active only if $t_{ir}=0$ and $t_{jr}=h_{ij}=1$, that is, if $x_j$ is a well-classified observation in the positive side of $\H_r$, while $x_i$  is wrong-classified in the negative side of $\H_r$ being $x_j$ the representative observation for $x_i$ (note that if $x_i$ is well-classified then $h_{ii}=1$ by \eqref{q7a} and then, the constraint cannot be activated). The second set of constraints, namely \eqref{q9}, can be analogously justified in terms of the negative side of $\H_r$. The main difference of these constraints with respect to \eqref{q4} and \eqref{q5} is that (\ref{q8}) and (\ref{q9}) are active only if $x_i$ is wrong-classified.

According to the above constraints, a misclassified observation $x_i$  is penalized in two ways with respect to each hyperplane $\mathcal{H}_r$. In case that $x_i$ is well-classified with respect to $\mathcal{H}_r$, but it belongs to the margin, then $e_{ir}=1 - \sign(\omega_{r}^t x_i + \omega_{r0}) (\omega_{r}^t x_i + \omega_{r0}) \leq 1$ and $d_{ir}=0$ ($t_{ir}=t_{jr}$). Otherwise, if $x_i$ is wrong-classified with respect to $\mathcal{H}_r$, then  $d_{ir}=1 - \sign(\omega_{r}^t x_i + \omega_{r0}) (\omega_{r}^t x_j + \omega_{r0}) \geq 1$ and $e_{ir}=0$ ($h_{ij}=1$ and $t_{ir}\neq t_{jr}$).

We illustrate the rationale of the proposed constraints on the data drawn in Figure \ref{fig::caso_2}. Observe that A is not correctly classified since it lies within  a cell in which the blue-class is not assigned. Suppose that B, a well-classified observation, is the representative of A ($h_{AB}=1$), then the model would have to penalize two types of errors. The first one with respect to $\H_2$. If we suppose $t_{B2}=1$, then $t_{A2}=0$, leading to an activation on constraint \eqref{q8} being $d_{A2}>0$. On the other hand, even though A is well-classified with respect to $\mathcal{H}_1$, we also have to penalize its margin violation. Again, if we assume $t_{B1}=1$, then $t_{A1}=1$, what would activate the constraint \eqref{q4} being $e_{A1}>0$.
\begin{figure}[h]
\begin{center}
\includegraphics[scale=0.5]{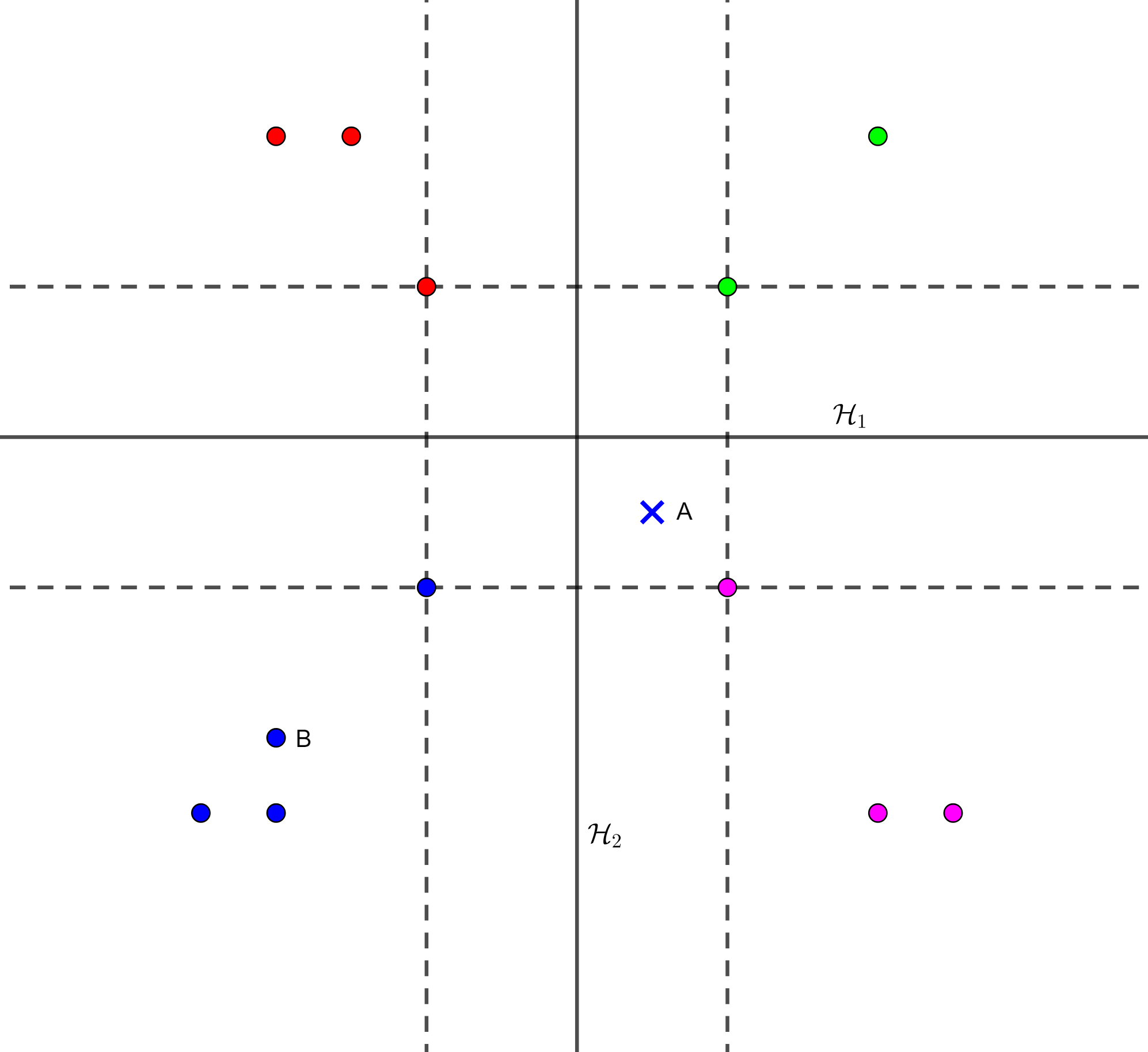}
\end{center}
\caption{Illustration of the in-margin and out-margin constraints of our model.\label{fig::caso_2}}
\end{figure}

The above comments can be summarized in the following mathematical programming formulation for { the problem}:
\begin{align}
\min & \;\; \|\omega_1\|^2+ C_1 \dsum_{i=1}^n \dsum_{r=1}^m e_{ir}+   C_2\dsum_{i=1}^n \dsum_{r=1}^m d_{ir} \label{model}\tag{${\rm MCSVM}$}\\
\mbox{s.t. } &\eqref{q1}-\eqref{q9}\nonumber\\
& \mathbf{\omega}_r \in \R^p, \omega_{r0} \in \R,&\forall r \in M,\nonumber\\
& d_{ir}, e_{ir} \geq 0,  t_{ir} \in \{0,1\}&\forall i\in N, r\in M,\nonumber\\
& h_{ij} \in \{0,1\}, &\forall i, j \in N,\nonumber\\
& z_{is} \in \{0,1\},& \forall i\in N, s\in K,\nonumber\\
&{ \xi_i \in \{0,1\},} & \forall i\in N.\nonumber
\end{align}

\eqref{model} is a mixed integer non linear programming model, whose nonlinear terms come from the norm minimization in the objective function { and constraints \eqref{q1}}, so that they are second order cone representable. { In case one chooses the $\ell_1$-norm instead of the Euclidean norm, the model becomes a mixed integer linear programming problem.} Therefore, the model is suitable to be solved using any of the available commercial solvers, as Gurobi, CPLEX, etc. The main bottleneck of the above formulation relies on the number $O(n^2)$ of binary variables.

\begin{rmk}[Ramp Loss misclassification errors]

An alternative measure of misclassification training errors is the \textit{ramp loss}. The ramp loss version of the model is interesting for certain instances since it allows one to improve the robustness against potential outliers. Instead of using out of margin hinge loss errors $h_{O}$, the ramp-loss measure consists of penalizing wrong-classified observations by a constant, independently on how far they are from being well-classified. Given an observation/label, $(\bar x, \bar y)$, the  ramp-loss with respect to $\mathbb{H}$, is defined as:
$$
{\rm RL}((\bar x, \bar y), \mathbb{H}) = \left\{\begin{array}{cl}
0 & \mbox{if $\bar x$ is well-classified} \\
1 & \mbox{otherwise}
\end{array}\right.
$$
Note that, for the training sample, the ramp-loss is represented in our model through the $\xi$-variables. More specifically, ${\rm RL}((x_i, y_i),\mathbb{H}) = \xi_i$ for all $i\in N$.  In order to do that we just need to introduce the following modifications on the MINLP problem:

\begin{align}
\min & \;\; \|\omega_1\|^2 + C_1\dsum_{i=1}^n \dsum_{r=1}^m e_{ir} +C_2 \dsum_{i=1}^n \xi_i\label{modelRL}\tag{${\rm MCSVM}_{\rm RL}$}\\
\mbox{s.t. } & \eqref{q1}-\eqref{q5}\nonumber\\
& \mathbf{\omega}_r \in \R^p, \omega_{r0} \in \R,& \forall r \in M,\nonumber\\
&  e_{ir} \geq 0, &\forall i\in N, r\in M,\nonumber\\
&h_{ij} \in \{0,1\}, &\forall i, j \in N,\nonumber\\
& z_{is} \in \{0,1\},& \forall i\in N, s\in K,\nonumber\\
& t_{ir} \in \{0,1\},&\forall i\in N, r\in M,\nonumber\\
& \xi_{i} \in \{0,1\}, &\forall i\in N.\nonumber
\end{align}

\end{rmk}

\subsection{Building the classification rule}

Recall that the main goal of multiclass classification is to determine a decision rule such that, given any observation, it is able to assign it a class, i.e., to determine the optimal suitable assignment. Hence, once the solution of \eqref{model} is obtained, the decision rule has to be derived. Given $x\in\R^p$, two different situations are possible: (a) $x$ belongs to a cell with an assigned class; and (b) $x$ belongs to a cell with no training observations inside, so with non assigned class. For the first case, $x$ is assigned to its cell's class. In the second case, different strategies to determine a class for $x$ are possible.

We propose the following assignment rule based on the same allocation methods used in \eqref{model}: observations are assigned to their closest well-classified representatives. More specifically,  let $\mathrm{s}(x)$ be the sign-pattern of $x$ with respect to the optimal arrangement of hyperplanes $\mathbb{H}^*  = \{(\omega_1^*, \omega^*_{10}), \ldots, (\omega_m^*, \omega^*_{m0})\}$ obtained from \eqref{model}, and let $J=\{j \in \{1, \ldots, n\}: \xi_j^*=0\}$ (here $\xi^*$ stand for the optimal vector obtained by solving \eqref{model}). Then, among all the well-classified observations in the training sample, $J$, we assign to $x$ the class of the one whose cell is the \textit{closest} (less separated from $x$). Such a classification of $x$ can be obtained by enumerating all the possible assignments, $O(|J|)$ and computing the distance measure over all of them. Equivalently, one can solve the following mathematical programming problem:

{\begin{align*}
\min  & \dsum_{j \in J}^n \dsum_{r=1\atop s(x_j)_r +\mathrm{s}(x)_r=0,}^{m} \gamma_j \, |(\omega^*_{r})^{t}x +\omega^*_{r0}|\\
\mbox{s.t. } & \dsum_{j \in J}^n \gamma_{j}=1,\\
&\gamma_{j} \in \{0,1\}, \forall j \in J
\end{align*}

\noindent where $\gamma_j=\left\{\begin{array}{cl} 1 & \mbox{if $x$ is assigned to the same cell as $x_j$,}\\
0 & \mbox{otherwise.}
\end{array}\right.$}

The integrality condition in the problem above can be relaxed, since the unique constraint in the problem is totally unimodular and thus, the problem is a linear programming problem. Clearly, the solution of the above problem gives the optimal labelling of $x$ with respect to the existing cells in the arrangement.

One could also consider other robust measures for such an assignment following the same paradigm, as min-max error or the like.

\providecommand{\ds}[0]{\displaystyle}

\subsection{Nonlinear Multiclass Classification}\label{subsec:kernel}

Finally, we analyze a crucial question in any SVM-based methodology, which is whether one can apply the Theory of Kernels in our framework. Using kernels means been able to map the observations (via some transformation $\varphi: \R^p \rightarrow \R^P$) to a higher dimensional space, where the separation of the data is more adequately performed. If the desired transformation, $\varphi$, is known, one could transform the data and solve the problem \eqref{model} with a higher number of variables. However, in binary SVMs, formulating the dual of the classification problem, one can observe that it only depends on the original data via the inner products of each pair of observations (originally in $\R^p$), i.e., through the amounts $x_i^t x_j$ for $i,j=1, \ldots, n$. If the  transformation $\varphi$ is applied to the data, the observations only appear in the (classical SVM) problem as $\varphi(x_i)^t \varphi(x_j)$ for $j=1, \ldots, n$. Thus, kernels are defined as generalized inner products as $K(a,b)=\varphi(a)^t \varphi(b)$ for each $a, b \in \R^p$, and they can be introduced using any of the well-known families of kernel functions (see e.g., \cite{horn}). Moreover, Mercer's theorem gives sufficient conditions for a function $K: \R^p \times \R^p \rightarrow \R$ to be a kernel function (one which is constructed as the inner product of a transformation of the features) what allows one to construct kernel measures that induce transformations. The main advantage of using kernels, apart from a probably better separation in the projected space, is that in binary SVM, the complexity of the \emph{transformed problem is the same as the  original one.} More specifically, the dual problems have the same structure and the same number of variables.

Although problem \eqref{model} is a MINLP, and then, duality results do not hold, one can apply decomposition techniques to separate the binary and the continuous variables and then, iterate over the binary variables by recursively solving certain continuous and easier problems (see e.g. Benders decomposition\cite{benders,geoffrion}...). The following result, whose proof can be found in the extended version of this paper (see \cite{BJP18}), states that our approach also allows us to find nonlinear classifiers via the kernel tools.

{This result is interesting by itself since links the general theory of nonlinear classifiers, very well-known for the standard SVM theory with Euclidean distance, to our multiclass framework. It is worth noting that for a function $h_H(\H_1, \ldots, \H_m) = \sum_{r=1}^m \|\omega_r\|^2$ the usual kernel trick construction applies \textit{mutatis-mutandis}. Nevertheless, as pointed out in Section \ref{ssec:sbc}, we elaborate our approach based on the natural measure of margin that maximizes the minimum separation between classes, namely $h_H(\H_1, \ldots, \H_m) = \max \{\|\omega_1\|^2,\ldots,\|\omega_m\|^2\}$. This change implies that the mathematical development known for the standard kernel trick does not carry over our new approach without a further analysis. We prove below that in this new framework one can also find nonlinear multiclass classifiers that, as in the standard SVM case, only depend on the transformation by  means of inner products of the original data. Hence, extending the kernel trick to this multiclass framework. }

\begin{thm}\label{th:kernels}
Let $\varphi: \R^p \rightarrow \R^P$ be a transformation of the feature space. Then, one can obtain a multiclass classifier which only depends on the original data by means of the inner products $\varphi(x_i)^t \varphi(x_j)$, for $i, j=1, \ldots, n$.
\end{thm}
\begin{proof}
See Appendix \ref{sec:kernels}.
\end{proof}


\section{A Math-Heuristic Algorithm}\label{sec:heuristic}

As mentioned above, the computational burden for solving \eqref{model}, that is a mixed integer non linear programming problem (in which the nonlinearities come from the norm minimization in the objective function), is the combination of the discrete aspects and the non-linearities  in the model. In this section we provide some heuristic strategies that allow us to cut down the computational effort  by fixing some of the variables. It will also provide good-quality initial feasible solutions when solving, exactly, \eqref{model} using a commercial solver.
Two different strategies are provided. The first one consists of applying a variable fixing strategy to reduce the number of $h$-variables in the model. Note that in principle, $n^2$ variables of this type are considered in the model. The second approach consists of fixing to zero some of the $z$-variables. These $nk$ variables allow us to model assignments between observations and classes. The proposed method is a math-heuristic approach, since after applying the adequate dimensionality reductions, Problem \eqref{model} (or \eqref{modelRL}) has to be solved. Also, although our strategies do not ensure any kind of optimality certificate, they produce a very good performance as will be shown in our computational experiments. Observe that when classifying datasets, the measure of the efficiency of a decision rule, as ours, is usually assessed by means of the accuracy of the classification on out-of-sample data, whereas the objective value of the proposed model is just an approximated measure of such an accuracy which cannot be computed only with the training data.

{
\begin{algorithm}[H]

\begin{enumerate}
\item Apply dimensionality reductions test based on algorithms \ref{alg1} and \ref{alg2}.
\item Find an initial solution generating $k$ separating hyperplanes.
\item Solve problem \eqref{model} (or \eqref{modelRL}) up to a prescribed accuracy for the train data.
\end{enumerate}

 \caption{A math-heuristic approach.\label{alg-mathheur}}
\end{algorithm}

In what follows we describe two strategies to reduce the dimensionality of the problem. These approaches are based on applying clustering techniques to the data. The methods are sensible to the number of clusters. For determining this parameter, we run a hierarchical clustering method, using as termination criterion a given squared Euclidean distance between the observations and their centroids.}

\subsection{Reducing the $h$-variables}

Our first strategy comes from the fact  that for a given observation $x_i$, there may be several possible choices for $h_{ij}$ to assume the value one with the same final result. Recall that  $h_{ij}$ could be equal to one whenever $x_j$ is a well-classified observation in the same class as $x_i$. The  errors $e_{ir}$ and $d_{ir}$ are then computed by using the class of $x_j$ but not the observation $x_j$ itself. Thus, if a set of well-classified observations of the same class is close enough,  only one of them can be the representative element of the group. In order to illustrate the procedure, we show in Figure \ref{fig:ex1} (left) a $4$-classes and $24$-points instance in which the classes are easily identified by applying any clustering strategy. In such a case \eqref{model} has ($24\times 24 =$) $576$ $h$-variables, but if we allow $h$ only to take value $1$ at a single point in each cluster, we obtain the same result but reducing to  $144$ ($24 \times 6 + 18$, where the 18 comes from the observation mentioned in the formulation in which each well-classified observation can be a representative element of itself) the number of variables. In Figure \ref{fig:ex1} (right),  we show the some clusters based on the data, and a (random) selection of a unique point at each cluster for which the $h$-values are allowed to be one.

\begin{figure}[h]
\begin{center}
\includegraphics[scale=0.3]{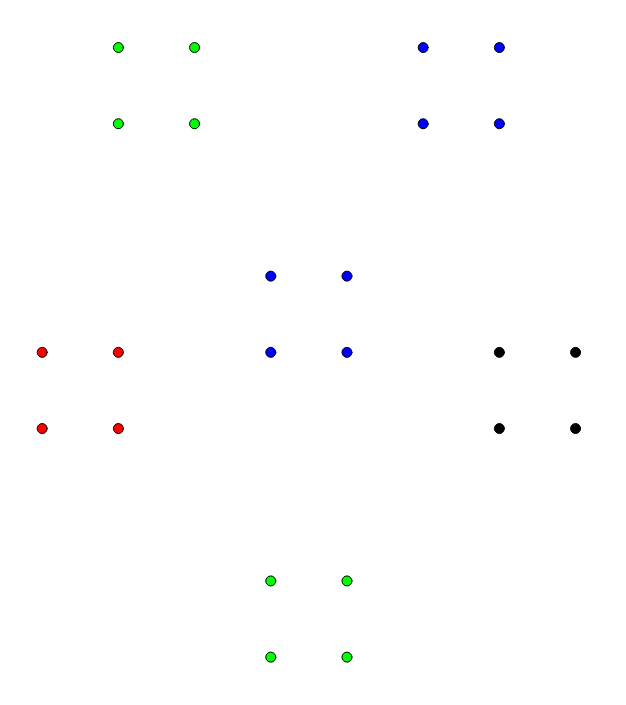} \hspace*{1cm} \includegraphics[scale=0.3]{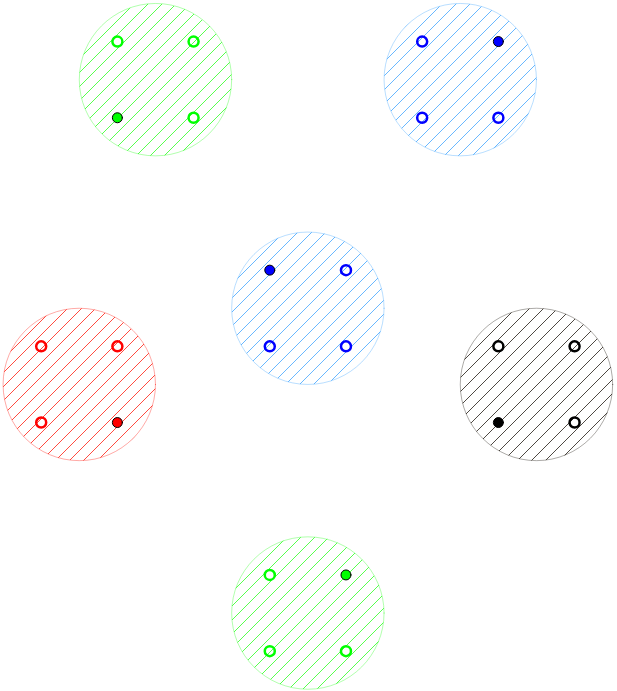}
\end{center}
\caption{Clustering observations for reducing $h$-variables.\label{fig:ex1}}
\end{figure}

This strategy is summarized in Algorithm \ref{alg1}.

\begin{algorithm}[H]

\begin{enumerate}
\item Cluster the dataset by \textit{approximated} classes: $\mathcal{C}_1, \ldots, \mathcal{C}_c$.
\item Randomly choose a single point at each cluster, $x_{i_j} \in \mathcal{C}_j$, for $j =1, \ldots, c$.
\item Set $h_{ij}=0$ for $j \not\in \{i_1, \ldots, i_c\}$.
\end{enumerate}

 \caption{Strategy to reduce $h$-variables.\label{alg1}}
\end{algorithm}

\subsection{Reducing the $z$-variables}

%
\begin{figure}[h]
\begin{center}
\includegraphics[scale=0.28]{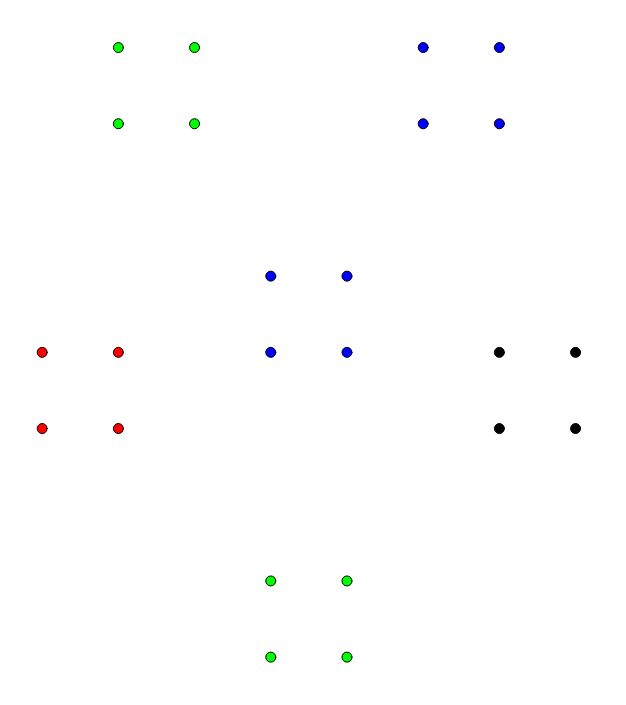}  \hspace*{1cm} \includegraphics[scale=0.27]{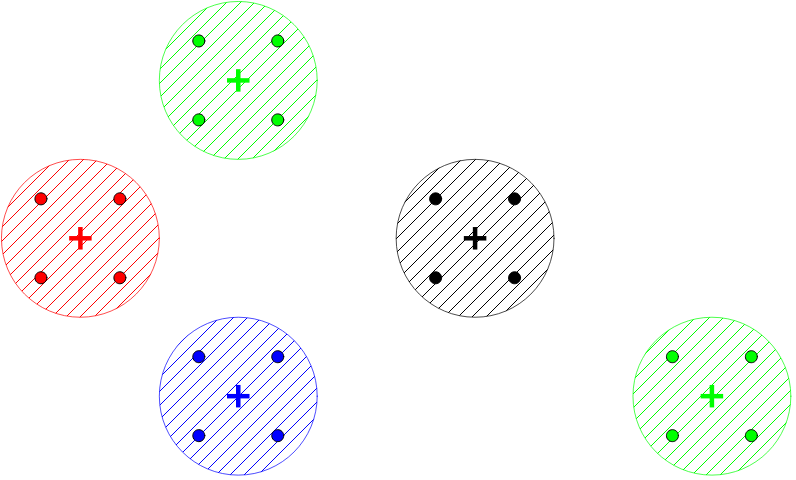}
\end{center}
\caption{Illustration of the strategy to reduce the $z$-variables.\label{fig:ex2}}
\end{figure}
%
%
%

{
The second strategy consists of fixing to zero some of the point-to-class assignments ($z$-variables). In the picture shown in Figure \ref{fig:ex2} (left), one can see a set of points which seems reasonable to group in 5 clusters. One may notice that assignments from the red class to the black class (and vice versa) are rarely going to occur following our approach. This is due to the fact that given this configuration of points, our model would provide a cell for red points located far from a black cell (otherwise it would probably not be maximizing the distance between classes). Following this idea, we derive a procedure to fix some of the $z$-variables to zero. Another observation that comes out from Figure \ref{fig:ex2}, is that with respect to the red cluster we obtain the following sorting on the set of distances: $d_{green}^{1} \leq d_{blue} \leq d_{black} \leq d_{green}^{2}$. Then, since $d_{green}^{1}<d_{green}^{2}$, we may not take into account the distance to the green cluster on the very right. Thus, we would fix to zero all $z_{is}$ variables that relate the red cluster with the maximum of their minimum distance set, that is, in this case we would fix to zero the $z_{is}$-variables associated to the black cluster with the red cluster ($d_{green}^{1}<d_{black}$ and $d_{blue}<d_{black}$) and vice versa.

The above observations lead us to some strategies for fixing $z$-variables to zero that we summarize in Algorithm \ref{alg2}.

\begin{algorithm}[H]

\begin{enumerate}
\item Group the observations in $L$ clusters, being each cluster formed by points of the same class.  Let $a_{\ell s}$ be the centroids of the cluster, $\ell=1,\ldots,L$, $s \in \left\lbrace 1, \ldots , k \right\rbrace$.
\item Compute the squared Euclidean distance matrix between centroids: $\mathcal{D}=\Big(\|a_{\ell s}- a_{qs'}\|^2\Big)$.
\item For each cluster $\ell$, $\ell=1,\ldots,L$, assigned to class $s$, compute the cluster $q$ with class $s'\neq s$ such that $\|a_{\ell s}-a_{qs'}\|^2$ is maximum and greater than a given threshold. For each observation $i$ in cluster $q$, set $z_{is}=0$. For each observation $\hat \i$ in cluster $\ell$, set $z_{\hat \i s'}=0$.
\end{enumerate}

 \caption{Strategy to reduce $z$-variables.\label{alg2}}
\end{algorithm}


}

\section{Experiments}\label{sec:experiments}

\subsection{Real Datasets}

In this section we report the results of our computational experience. We have run a series of experiments to analyze the performance of our model in some real  datasets widely used in the classification literature, and that are available in the UCI machine learning repository \cite{uci}. Summarized information about the datasets is detailed in Table \ref{datos}. In such a table we report, for each dataset, the number of observations considered in the training sample ($n_{\rm {Tr}}$) and test sample ($n_{\rm {Te}}$), the number of features ($p$), the number of classes ($k$), the number of hyperplanes used in our separation ($m$), and the number of hyperplanes required by the OVO methodology ($m_{\rm OVO}$).

\begin{table}[h]\label{datos}
\centering
\begin{tabular}{l|ccccccc}
{\bf Dataset} & $n_{\rm {Tr}}$ & $n_{\rm Te}$  & $p$  & $k$&  & $m$ &  $m_{\rm OVO}$\\
\hline\hline

\texttt{Forest }& 75 & 448 & 28 & 4 & & 3 & 6 \\
\texttt{Glass}  & 75 & 139 & 10 & 6 & & 6 & 15 \\
\texttt{Iris}  & 75 & 75  & 4 & 3 &  & 2 & 3  \\
\texttt{Seeds} & 75 & 135 & 7 & 3 & & 2 & 3 \\
\texttt{Wine} & 75 & 103  & 13 &  3 & & 2 & 3 \\
\texttt{Zoo }& 75 & 26 & 17 & 7 & & 4 & 21 \\
 \hline

\end{tabular}
\caption{Data sets used in our computational experiments.\label{datos}}
\end{table}

For these datasets, we have run both the hinge-loss \eqref{model} and the ramp-loss \eqref{modelRL} models, measuring the margin with the $\ell_1$ and the $\ell_2$ norms.
We have performed a $5$-cross validation scheme to test each of the approaches. Thus, the data sets were split into 5 train-test random partitions. Then, the models were solved for the training sample and the resulting classification rule was applied to the test sample.  We report the average accuracy, ACC, in percentage, of the $5$ repetitions of the experiment on test:
$$
{\rm ACC} = \dfrac{\# \text{Well Classified Test Observations}}{n_{\rm {Te}}} \cdot 100.
$$

The parameters of our models were also chosen after applying a grid-based $4$-cross validation scheme. In particular, we calibrate the value of $m$ (number of hyperplanes to be located)  and the misclassification costs $C_1$ and $C_2$ in:
$$
m \in \{ 2, \ldots , k  \}, \quad C_1, C_2 \in \{ 0.1, 0.5 , 1, 5, 10\}.
$$

For hinge-loss models $C_1=C_2$, whereas for ramp-loss models we consider $C_1 < C_2$  to give a high penalty to wrong-classified observations.
As a result we obtain a misclassification error for each grid point, and we choose the combination of parameters that provide the lowest error. The same methodology was also applied to the other methods: OVO, Weston-Watkins (WW), Crammer-Singer (CS) and Lee, Lin and Wahba (LLW), calibrating the misclassifying cost $C$ in $\{ 10^i, \ i =-6,\ldots , 6 \}$.

The mathematical programming models solving the MCSVM methods were coded in Python 3.6, and solved using Gurobi 7.5.2 on a PC Intel Core i7-7700 processor at 2.81 GHz and 16GB of RAM. The standard methods (OVO, WW and CS) were applied using R-KernLab. Finally, LLW was applied using the software package MSVMpack provided by  \cite{MSVMpack}. \\

In Table \ref{tabla2} we report the average accuracies obtained with our 4 models: (\eqref{model} and \eqref{modelRL} with $\ell_1$ and $\ell_2$ norms) and  those obtained with OVO, WW, CS and LLW. The first two columns ($\ell_1$ RL and $\ell_1$ HL) show the average accuracies of our two approaches (Ramp Loss - RL- and Hing Loss -HL-) using the $\ell_1$-norm. On the other hand,  the third and four columns ($\ell_2$ RL and $\ell_2$ HL) provide the same results for the $\ell_2$-norm. In the last four columns, we report the average accuracies obtained with the   OVO, WW, CS and LLW methods.
The best accuracies obtained for each dataset are bolfaced in Table \ref{tabla2}.

One can observe that our methods always outperform the results obtained by OVO, WW and CS, although the results are rather similar. Actually, running the two samples proportion test among them, we can not ensure significative differences in all cases. Comparing our  methods with LLW the results are different. In three out of the 6 databases (Forest, Glass and Iris) our methods are superior to LLW with up to 10\% significative differences with respect to the two samples proportion tests. In the remaining three databases (Seeds, Wine and Zoo) the results are similar with no statistical significative differences with respect to the two samples proportion test.

The results indicate that these UCI databases are friendly for linear classifiers (with the only exception of Glass) and thus all these methods perform reasonably well on test prediction. Thus, it is not possible to establish  a clear ranking of these classification methods based only on these databases. In order to asses a more complete comparative of the methods we continue the analysis in the following subsection with a battery of more complex datasets.

\begin{table}[h]
\centering
\begin{tabular}{l|cccccccccc}
Dataset & $\ell_1$ RL & $\ell_1$ HL   & $\ell_2$ RL  & $\ell_2$ HL  &  &  &  OVO & WW & CS  & LLW\\
\hline
\hline
Forest &  80.66 & 80.12  &  {\bf 82.30} &  81.62  & &  & 82.10 & 78.40 & 78.60 & 72.54 \\
Glass  & 64.92  &  64.92  &  {\bf 65.32}  &  {\bf 65.32} & &  & 58.76 & 56.25 & 59.26& 57.04\\
Iris  &  95.08 &  95.40  &  96.44 &  {\bf 96.66} &  &  & 93.80 & 96.44  & 96.44 & 84.17\\
Seeds &   93.66  &   93.66 &  93.52 &  93.52 & &  & 91.02 & 93.52 & 93.52& {\bf 95.46}\\
Wine & 95.20 & 95.20  &  {\bf 96.82} &   {\bf 96.82} & &  &  96.34 & 96.09 & 96.17& 96.31 \\
Zoo &  89.75 & 89.75 &  89.75 &   89.75 & &  & 87.44 & 87.68 & 87.68 & {\bf 91.53}\\
\hline
\end{tabular}
\caption{Average accuracies obtained for the real-world instances \label{tabla2}}
\end{table}

%

\subsection{Synthetic Experiments}

This section reports extra computational experiments over some synthetic instances that allow us to establish some rank of the methods based on their accuracies. We have generated $6$ instances of $750$ observations in $\R^{10}$ distributed as  multivariate normal distributions separated by a constant factor. The instances are denoted as $X$\texttt{C}$Y$\texttt{N} where $X$ is the number of classes (ranging in $\{2, 3,4,7,10\}$) and $Y$ the number of different multivariate normal distributions (ranging in $\{4,6,8,15,20\}$). All the instances are available at \url{http://bit.ly/SynthData_MCSVM} for benchmarking purposes. Observe that for each instance, the class labels  have been randomly assigned to the normal distributions. For illustration purposes, a two-dimensional instance generated in the same way that our $10$-dimensional instances  is shown in Figure \ref{fig:clouds}: the data are generated  according to $20$  normal distributions which are assigned to $10$ classes.


\begin{figure}[h]
\includegraphics[scale=0.75]{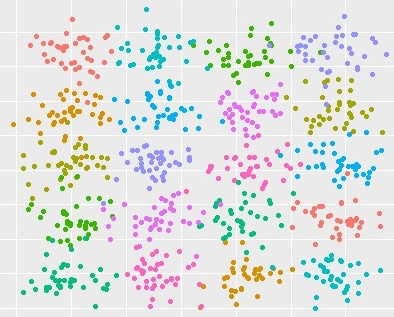}
\centering
\caption{A $2$-dimensional illustration of our instances.\label{fig:clouds}}
\end{figure}

In Table \ref{table4} we report the average accuracies obtained with a $10$-fold cross validation experiment, in which $75$ observations are taken into the training samples and $675$ in the test samples. As before, we have compared our approach (with the Euclidean norm and Hinge-Loss misclassification error) with the existing methodologies: OVO, WW, CS and LLW.  The calibration of the parameters was also done as in the previous section.
\begin{table}[h]
\centering
\begin{tabular}{l|cc|cccc}
Dataset & $\ell_2$ HL &$m$  &   OVO & WW & CS  &LLW  \\
\hline\hline
\texttt{2C4N}      &     {\bf 94.35}&  2   &        60.75 &  60.75     &   60.75    & 60.75 \\
\texttt{3C6N}      &     {\bf 85.74}&  3   &        39.47 & 41.69 & 39.03 & 36.50 \\
\texttt{4C8N\_1}      &     {\bf 92.76}&  4   &        36.46 & 32.37 & 29.14 & 31.86 \\
\texttt{4C8N\_2}      &     {\bf 91.78}&  4    &        48.54 & 35.14 & 34.69 & 39.14 \\
\texttt{7C15N}    &     {\bf 88.54}&  6   &        27.37 & 19.64 & 18.63 & 20.35 \\
\texttt{10C20N}  &     {\bf 85.81} &  7     &        29.73 & 16.17 & 15.37 & 15.10 \\
\end{tabular}
\caption{Average accuracies obtained for  the synthetic instances. \label{table4}}
\end{table}

One can observe in Table \ref{table4} that the results obtained with our approach are much better than those obtained with the other approaches. The generation procedure permits that, in the synthetic instances,  separated clouds of points are assigned to the same class. As it can be anticipated, our methodology adapts well to this characteristic whereas the other approaches  fail to handle these data. The reader may observe that this type of data are common in real-world datasets. In particular, many diseases are associated to low or high values of certain medical indices thus fitting to this topology in which separated clusters of observations belong to the same class.

Our main conclusion, from the results reported in Table \ref{table4}, is that our method is adequate for this type of synthetic data, highly outperforming OVO, WW, CS and LLW. Moreover, the accuracy percentages are not only superior but  they are also statistically better with respect to the two samples proportion test with a significance level of $1\%$.

\section{Conclusions}
In this paper we propose a novel modeling framework for multiclass classification based on the Support Vector Machine paradigm, in which the different classes are linearly separated and the separation between classes is maximized. We propose two approaches, that depend on the way to account for the misclassification error, to compute an optimal arrangement of hyperplanes subdividing the space into cells, and so that each cell is assigned to a class based on the training data. The models result in Mixed Integer (Linear and Non Linear) Programming problems. Some dimensionality reduction and preprocessing strategies are presented in order to help solvers to find good (optimal) solutions of the corresponding problems. We also prove that an analogous of the kernel trick can be extended to this framework. The performance of this approach is illustrated on some well-known datasets of the multi-category classification literature as well as in some synthetic, but still realistic, examples, in which our approach works remarkably well compared to the existing methodologies. Several extensions of our approach are possible. Among them we would like to mention the use of  heuristic algorithms to solve the complex mixed integer nonlinear programs which may alleviate the computational  burden of the methodology still keeping high quality solutions. Moreover, our approach could also be extended to the framework of semisupervised learning~\cite{bennet-demiriz,ortigosa} by assigning unlabelled observations to their closest well-classified cells (which are obtained using the labeled training data). Both research lines seem to be promising and will be the focus of a forthcoming paper.


\appendix

\section*{Proof of Theorem \ref{th:kernels}}\label{sec:kernels}

Let us consider the mathematical programming formulation for our problem:

\begin{align}
\min & \;\; \|\omega_1\|^2+ C_1 \dsum_{i=1}^n \dsum_{r=1}^m e_{ir}+   C_2\dsum_{i=1}^n \dsum_{r=1}^m d_{ir} \label{model}\tag{${\rm MCSVM}$}\\
\mbox{s.t. } &\frac{1}{2}\|\omega_{r-1}\|^2 \geq \frac{1}{2}\|\omega_r\|^2, \forall r=2, \ldots, m,\label{q1}\\
&\omega_{r}^tx_i + w_{r_0} \geq - T (1-t_{ir}), \forall i\in N, r \in M,\label{t+}\\
& \omega_{r}^tx_i + w_{r_0}  \leq T t_{ir}, \forall i\in N, r \in M,\label{t-}\\
&\dsum_{s=1}^k z_{is}=1, \forall i\in N,\label{q2}\\
& \|z_i - z_j\|_1 \leq 2 \|t_i-t_j\|_1, \forall i, j\in N,\label{q3}\\
& \xi_i \geq  \frac{1}{2}  \|z_i-\delta_{i}\|_1,  \forall i\in N,\label{xi}\\
& \dsum_{j \in N:\atop y_i=y_j} h_{ij}=1, \forall i\in N,\label{q6}\\
& \xi_j+h_{ij} \leq 1, \forall i, j\in N (y_i=y_j),\label{q7}\\
& h_{ii} = 1- \xi_i, \forall i\in N,\label{q7a}\\
&\omega_{r}^t x_i + \omega_{r0} \geq 1-e_{ir}- T\;(3-t_{ir}-t_{jr}-h_{ij}), \forall r \in M,\label{q4}\\
& \omega_{r}^t x_i +\omega_{r0} \leq -1 +e_{ir}+ T\; (1+t_{ir}+t_{jr}-h_{ij}), \forall r\in M,\label{q5}\\
& d_{ir} \geq 1- \omega^t_r x_i -\omega_{r0} - T\,(2+t_{ir}-t_{jr}-h_{ij}), \forall i,j\in N (y_i=y_j), r\in M,\label{q8}\\
& d_{ir} \geq 1+ \omega^t_r x_i + \omega_{r0} - T\,(2-t_{ir}+t_{jr}-h_{ij}), \forall i,j\in N (y_i=y_j), r\in M,\label{q9}\\
& \mathbf{\omega}_r \in \R^p, \omega_{r0} \in \R, \forall r \in M,\nonumber\\
& d_{ir}, e_{ir} \geq 0,  t_{ir} \in \{0,1\}, \forall i\in N, r\in M,\nonumber\\
& h_{ij} \in \{0,1\}, \forall i, j \in N,\nonumber\\
& z_{is} \in \{0,1\}, \forall i\in N, s\in K,\nonumber\\
&\xi_i \in \{0,1\},  \forall i\in N.\nonumber
\end{align}

Note that once the binary variables of our model are fixed, the problem becomes  polynomial time solvable and it reduces to find the coordinates of the coefficients and intercepts of the hyperplanes and the different misclassifying errors. In particular, it is clear that the MINLP formulation for the problem is equivalent to:

\begin{align*}
\displaystyle \min_{h,z,t,\xi} & \;\; \Phi(h,z,t,\xi)\\
\mbox{s.t. } &\eqref{q2}-\eqref{q7a},\\
& h_{ij} \in \{0,1\}, &\forall i, j \in N,\nonumber\\
& t_{ir} \in \{0,1\}&\forall i\in N, r\in M,\nonumber\\
& z_{is} \in \{0,1\},& \forall i\in N, s\in K,\nonumber\\
& \xi_{i} \in \{0,1\}, &\forall i\in N.\nonumber
\end{align*}

\noindent where $\Phi$ is the evaluation of the margin and hinge-loss errors for any  assignment provided by the binary variables. That is,

\begin{align}
\Phi( h,  z,  t,  \xi) = & \min_{\omega,\omega_0,e,d} \;\; \dfrac{1}{2}\|\omega_1\|^2+ C_1 \dsum_{i=1}^n \dsum_{r=1}^m e_{ir}+  C_2\dsum_{i=1}^n \dsum_{r=1}^m d_{ir} \label{evalfi}  \tag{${\rm Eval}_\Phi$} \nonumber\\
&  \mbox{s.t. }   \eqref{q1}, \eqref{q4}-\eqref{q9}, \nonumber\\
& \mathbf{\omega}_r \in \R^d, \omega_{r0} \in \R,&\forall r \in M,\nonumber\\
& d_{ir}, e_{ir} \geq 0 &\forall i\in N, r\in M.\nonumber
\end{align}
The above problem would be separable provided that the first $m-1$ constraints \eqref{q1} were relaxed. For the sake of simplicity in the notation, we consider the following functions, $\kappa^\ell: \{0,1\} \rightarrow \R$ for $\ell=1, 2, 3$, defined as:
$$
\kappa^1(t) := T(1-t), \;\; \kappa^2(t) := T t, \;\; \kappa^3(t) := - 1 + Tt
$$
for $t\in \{0,1\}$. Note that $\kappa^1(0)=\kappa^2(1)=T$, $\kappa^1(1)=\kappa^2(0)=0$, and that $\kappa^3(0)=-1$, $\kappa^3(1)=T-1$.

Based on the separability mentioned above, we introduce another instrumental family of problems for all $r=2,\ldots,m$, namely,
\begin{align}
&\Phi_r( h, z, t, \xi,\omega_1) &=  \min_{\omega_r, \omega_{r0},  e, d}   C_1 \dsum_{i=1}^n  e_{ir}+  C_2\dsum_{i=1}^n  d_{ir}&\nonumber\\
&   &\mbox{s.t. }    \dfrac{1}{2} \|\omega_r\|^2 - \dfrac{1}{2} \|\omega_1\|^2 &\leq 0, \nonumber\\
&&-\omega_{ir}^tx_i - \omega_{r0} &\leq \kappa^1(t_{ir}), \forall i,r,\nonumber\\
&& \omega_{ir}^tx_i + \omega_{r0}  &\leq  \kappa^2(t_{ir}), \forall i, r,\nonumber\\
&&-\omega_{r}^t x_i - \omega_{r0} - e_{ir} &\leq  \kappa^3(u_{ijr}^+), \forall i,j,r\nonumber\\
& &\omega_{r}^t x_i +\omega_{r0} -e_{ir} &\leq  \kappa^3(u_{ijr}^-), \forall i,j, r, \nonumber\\
& & -\omega^t_r x_i -\omega_{r0} -d_{ir}&\leq \kappa^3(q_{ijr}^+),\forall i,j (y_i=y_j), \label{spr}\tag{${\rm SP}_r$}\nonumber\\
& & \omega^t_r x_i + \omega_{r0}-d_{ir}  &\leq \kappa^3(q_{ijr}^-), \forall i,j (y_i=y_j), \nonumber\\
& &-d_{ir} &\leq 0, \forall i, \nonumber\\
& &-e_{ir}&\leq 0, \forall i,\nonumber\\
& & \mathbf{\omega}_r \in \R^d,& \omega_{r0} \in \R,\nonumber
\end{align}
where for simplifying the notation we have introduced the auxiliary variables $u_{ijr}^+:= 3-t_{ir}-t_{jr}-h_{ij}$, $u_{ijr}^-:=1+t_{ir}+t_{jr}-h_{ij}$, $q_{ijr}^+ = 2+t_{ir}-h_{ij}-t_{jr}$ and $q_{ijr}^- = 2+t_{jr}-h_{ij}-t_{ir}$, for $i, j \in N$ and $r\in M$.

Observe that $\Phi_r$, apart from the first constraint, only considers variables associated to the $r$th hyperplane.

Moreover, we need another problem that accounts for the first part of $\Phi$.
\begin{align}
\Phi_1( h, z, t, \xi) &=  \min_{\omega_1, \omega_{10},  e, d}  \dfrac{1}{2} \|\omega_1\|_1^2 +C_1 \dsum_{i=1}^n e_{i1}+  C_2\dsum_{i=1}^n  d_{i1}\nonumber\\
 & \mbox{s.t. } -\omega_{1}^tx_i - \omega_{10} \leq \kappa^1(t_{i1}), \forall i,\nonumber\\
& \omega_{1}^tx_i + \omega_{10}  \leq  \kappa^2(t_{i1}), \forall i, \nonumber\\
&-\omega_{1}^t x_i - \omega_{10} - e_{i1} \leq  \kappa^3(u_{ij1}^+), \forall i,j,\nonumber\\
 &\omega_{1}^t x_i +\omega_{10} -e_{i1} \leq  \kappa^3(u_{ij1}^-), \forall i,j, r, \nonumber\\
 & -\omega^t_1 x_i -\omega_{10} -d_{i1}\leq \kappa^3(q_{ij1}^+),\forall i,j (y_i=y_j), \label{sp1}\tag{${\rm SP}_1$}\nonumber\\
 & \omega^t_1 x_i + \omega_{10}-d_{i1}  \leq \kappa^3(q_{ij1}^-), \forall i,j (y_i=y_j), \nonumber\\& -d_{i1} \leq 0, \forall i, \nonumber\\
& -e_{i1}\leq 0, \forall i,\nonumber\\
&  \mathbf{\omega}_1 \in \R^d, \omega_{10} \in \R,\nonumber
\end{align}

Thus, using the above notation, \eqref{model} is equivalent to the following problem:

\begin{align*}
\displaystyle \min_{h,z,t,\xi,\omega_1} & \;\; \Phi_1(h,z,t,\xi)+\sum_{r=2}^m \Phi_r(h,z,t,\xi,\omega_1)\\
\mbox{s.t. } &\eqref{q2}-\eqref{q7a}\\
& h_{ij} \in \{0,1\}, &\forall i, j \in N,\nonumber\\
& t_{ir} \in \{0,1\}&\forall i\in N, r\in M,\nonumber\\
& z_{is} \in \{0,1\},& \forall i\in N, s\in K,\nonumber\\
& \xi_{i} \in \{0,1\}, &\forall i\in N.\nonumber
\end{align*}

Observe that the above problem only accounts for $\omega_1$ and the binary variables. Once they are fixed can be plugged into the \eqref{spr}, $r=1,\dots,m$ subproblems and then it allows one to find the optimal values of the continuous variables. The elements $\kappa^1( t_{ir}), \kappa^2( t_{ir}) ,\kappa^3(u_{ijr}^+) , \kappa^3(u_{ijr}^-), \kappa^3( q^+_{ijr})$, and $\kappa^3( q^-_{ijr})$ are fixed constants once the binary variables are fixed.

In order to solve the problem for a fixed set of binary variables we can proceed recursively, solving independently \eqref{spr} for all $r=2,\ldots,m$, and then combining their solutions with \eqref{sp1}. 

Therefore, to get that goal we apply Lagrangean duality to obtain an exact dual reformulation. Indeed, relaxing the constraints of \eqref{spr} with dual multipliers $\mu_r\ge 0$ (assuming that $\mu^0_r \neq 0$) and denoting by $\alpha_{ir} =   \alpha_{ir} = \mu_{ir}^1- \mu_{ir}^2 +\dsum_{j: y_i=y_j} (\mu_{ijr}^3- \mu_{ijr}^4+\mu_{ijr}^5 - \mu_{ijr}^6)$, for all $i=1, \ldots, n$ and $r =1, \ldots, m$; after some derivations that can be followed in Lemma \ref{le:A1}, one can check that evaluating $\Phi$ for given values of the binary variables, can be obtained solving the continuous optimization problem \eqref{jspf}. (All details can be found in Lemma \ref{le:A1}.)

Next, we analyze how the evaluation of the function $\Phi$, given in problem \eqref{evalfi}, depends on the original data. In order to do that we find the optimal solutions of \eqref{jspf}. Dualizing the constraints that correspond to $\Phi_1$ with multipliers $\mu_1\ge 0$ and those where the variables $\Gamma_r$ appear,  with dual multipliers $\gamma_r$,  the Lagrangean function of the problem \eqref{jspf} results in:

\begin{eqnarray*}
\begin{split}
\mathcal{L}(\omega_1, \omega_{01}, d, e; \mathbf{\mu}) &= \frac{1}{2} \|\omega_1\|^2(1-\sum_{r=2}^m \gamma_r \mu^0_r) + C_1\sum_{i=1}^n e_{i1}+C_2\sum_{i=1}^n d_{i1}+ \sum_{r=2}^m \Gamma_r(1-\gamma_r) \nonumber \\
& +\sum_{r=2}^m \left( \dfrac{\gamma_r}{2\mu^0_r} \dsum_{i,j=1}^n \alpha_{ir}\alpha_{jr} x_i^t x_j  +\dsum_{i,j: u^+_{ijr}=0} \mu_{ijr}^3 \right.\nonumber\\
&\left.+ \dsum_{i,j: u^-_{ijr}=0} \mu_{ijr}^4 +\dsum_{i,j: q^+_{ijr}=0} \mu_{ijr}^5+ \dsum_{i,j: q^-_{ijr}} \mu_{ijr}^6\right) \\
&+\dsum_{i} \mu^1_{i1} \Big(-\omega_{1}^t x_i - \omega_{10}- \kappa^1(t_{i1})\Big)\\
&+\dsum_{i} \mu^2_{i1} \Big(\omega_{1}^t x_i + \omega_{10}- \kappa^2(t_{i1})\Big)\\
&  +\dsum_{i} \mu^3_{ij1} \Big(-\omega_{1}^t x_i - \omega_{10} - e_{i1} - \kappa^3(\bar u^+_{ij1})\Big)\\
&+ \dsum_{i} \mu^4_{ij1} \Big( \omega_{1}^t x_i +\omega_{10} -e_{i1}  - \kappa^3(\bar u^-_{ij1})\Big)
\end{split}
\end{eqnarray*}
\begin{eqnarray*}
\begin{split}
&+ \dsum_{i, j} \mu^5_{ij1} \Big(-d_{ir} -\omega^t_1x_i -\omega_{10}- \kappa^3(\bar q_{ij1}^+)\Big)\\
& +\dsum_{i, j} \mu^6_{ij1} \Big(-d_{ir} + \omega^t_1 x_i + \omega_{10}- \kappa^3(\bar q_{ij1}^-)\Big)\\
& + \dsum_{i} \mu^7_{i1}(-d_{i1}) +  \dsum_{i} \mu^8_{i1}(-e_{i1})   
\end{split}
\end{eqnarray*}

and its KKT optimality conditions reduce to:
\begin{itemize}
\item[$\bullet$] $(1-\sum_{r=2}^m \gamma_r \mu^0_r)\omega_1= \dsum_{i=1}^n \Big (\mu_{i1}^1- \mu_{i1}^2 + \dsum_{j: y_i=y_j} (\mu_{ij1}^3 - \mu_{ij1}^4+\mu_{ij1}^5 - \mu_{ij1}^6 )\Big) x_i$.
\item[$\bullet$] $\dsum_{i=1}^n \Big (\mu_{i1}^1- \mu_{i1}^2 + \dsum_{j: y_i=y_j} (\mu_{ij1}^3 - \mu_{ij1}^4+\mu_{ij1}^5 - \mu_{ij1}^6 )\Big) =0$.
\item[$\bullet$] $\dsum_{j: y_i=y_j}(\mu_{ij1}^3 +\mu_{ij1}^4) + \mu^8_{i1} = C_1$, for all $i$.
\item[$\bullet$] $ \dsum_{j: y_i=y_j} (\mu_{ijr}^5 + \mu_{ijr}^6 ) + \mu^7_{ir} = C_2$, for all $i$.
\item[$\bullet$] $1-\gamma_r=0$, for all $r=2,\ldots,m$.
\item[$\bullet$] $\mu \geq 0$, $\gamma \ge 0$.
\end{itemize}

Using the same notation as before,  $\alpha_{i1} =  \mu_{i1}^1- \mu_{i1}^2 + \dsum_{j: y_i=y_j} (\mu_{ij1}^3 - \mu_{ij1}^4+\mu_{ij1}^5 - \mu_{ij1}^6 )$, for all $i=1, \ldots, n$, and simplifying the expressions using the KKT conditions and the complementary slackness conditions we get that the strong dual of \eqref{jspf} is:

\begin{align}
\max_{\omega, \omega_{.0}, d, e; \mathbf{\mu}} &  \dfrac{1}{2(1-\sum_{r=2}^m \mu^0_r)} \dsum_{i,j=1}^n \alpha_{i1}\alpha_{j1} x_i^t x_j  +\dsum_{i,j: u^+_{ij1}=0} \mu_{ijr}^3 + \dsum_{i,j: u^-_{ij1}=0} \mu_{ij1}^4\nonumber\\
 &+\dsum_{i,j: q^+_{ij1}=0} \mu_{ijr}^5+ \dsum_{i,j: q^-_{ij1}=0} \mu_{ij1}^6 +\sum_{r=2}^m \left( \dfrac{1}{2\mu^0_r} \dsum_{i,j=1}^n \alpha_{ir}\alpha_{jr} x_i^t x_j    \right.\nonumber \\
 & \left.+\dsum_{i,j: u^+_{ijr}=0} \mu_{ijr}^3 + \dsum_{i,j: u^-_{ijr}=0} \mu_{ijr}^4+\dsum_{i,j: q^+_{ijr}=0} \mu_{ijr}^5+ \dsum_{i,j: q^-_{ijr}=0} \mu_{ijr}^6\right)\nonumber\\
\mbox{s.t. }& \dsum_{i=1}^n \alpha_{ir} =0, \forall r\label{djspf}\tag{${\rm DJSP}$}\\
& \dsum_{i=1}^n \alpha_{i1} x_i = (1-\sum_{r=2}^m \mu^0_r) \omega_1,\nonumber\\
& \dsum_{i=1}^n \alpha_{ir} x_i = \mu^0_r \omega_r, \forall r\ge 2 \nonumber\\
&\dsum_{j: u^+_{ijr}=0} \mu_{ijr}^3 + \dsum_{j: u^-_{ijr}=0} \mu_{ijr}^4   \leq C_1, \forall i,\nonumber\\
&\dsum_{j: q^+_{ijr}=0} \mu_{ijr}^5+ \dsum_{j: q^-_{ijr}=0} \mu_{ijr}^6 \leq C_2, \forall i, r,\nonumber\\
& \mu_{ir}^1 =0, \forall i, r \mbox{ with } \bar t_{ir}=0,\nonumber\\
& \mu_{ir}^2 =0, \forall i, r \mbox{ with } \bar t_{ir}=1,\nonumber\\
& \alpha_{ir} =  \mu_{ir}^1- \mu_{ir}^2 +\dsum_{i,j: u^+_{ij1}=0} \mu_{ijr}^3- \dsum_{i,j: u^-_{ij1}=0} \mu_{ijr}^4+\dsum_{i,j: q^-+{ij1}=0} \mu_{ijr}^5\\
& - \dsum_{i,j: q^-_{ij1}=0}  \mu_{ijr}^6,\forall i, \nonumber\\
& \mathbf{\mu_r} \geq 0,  \mathbf{\omega_r} \in \R^d, \omega_{r0} \in \R,\; \forall r=1,\ldots,m \nonumber
\end{align}

Note that the objective function of \eqref{djspf} only depends of the $x$-input data through the inner products $x_i^tx_j$ for $i, j=1, \ldots, n$, and also the expressions of $\omega_1$  from the dual variables is given as:
$$
(1-\sum_{r=2}^m \mu^0_r) \omega_1= \dsum_{i=1}^n \alpha_{i1} x_i.
$$
The dependence of $\omega_1$ with other $\omega_r$ in the primal formulation is only through the nonincreasing sorted values of $\|\omega_1\|, \ldots, \|\omega_m\|$ in which we now that the largest value is $\|\omega_1\|$. Thus, solving the dual problem \eqref{djspf} allows us to determine all the optimal hyperplanes $\omega_r$, for all $r=1,\ldots,m$. In case a transformation $\varphi$ is performed to the input data, the dependence of the data in the problem will be through the inner products $\varphi(x_i)^t \varphi(x_j)$ for all $i, j=1, \ldots, n$, and the kernel Theory can be applied.

{\hfill $\square$}

\begin{lem} \label{le:A1}
The evaluation of $\Phi$  for given values of the binary variables $\bar t$, $\bar \xi$, $\bar h$ and $\bar z$ can be obtained solving the following continuous optimization problem:

\begin{align}
\hat\Phi( h, z, t, \xi) = & \min \left\{ \frac{1}{2} \|\omega_1\|^2 + C_1\sum_{i=1}^n e_{i1}+C_2\sum_{i=1}^n d_{i1}+ \sum_{r=2}^m \Gamma_r  \right\}\nonumber\\
& \mbox{s.t. } \left(-\frac{\mu^0_r}{2} \|\omega_1 \|^2 + \dfrac{1}{2\mu^0_r} \dsum_{i,j=1}^n \alpha_{ir}\alpha_{jr} x_i^t x_j\right.\nonumber \\
 &+\dsum_{i,j: u^+_{ijr}=0} \mu_{ijr}^3 + \dsum_{i,j: u^-_{ijr}=0} \mu_{ijr}^4\nonumber\\
 &\left. +\dsum_{i,j: q^+_{ijr}=0} \mu_{ijr}^5+ \dsum_{i,j: q^-_{ijr}=0} \mu_{ijr}^6+C_1\right) \le \Gamma_r, \forall r\ge 2, \nonumber \\
&-\omega_{1}^tx_i - \omega_{10} \leq \kappa^1(t_{i1}), \forall i,\nonumber\\
& \omega_{1}^tx_i + \omega_{10}  \leq  \kappa^2(t_{i1}), \forall i, \nonumber\\
&-\omega_{1}^t x_i - \omega_{10} - e_{i1} \leq  \kappa^3(u_{ij1}^+), \forall i,j,\nonumber\\
 &\omega_{1}^t x_i +\omega_{10} -e_{i1} \leq  \kappa^3(u_{ij1}^-), \forall i,j, r, \nonumber\\
 & -\omega^t_1 x_i -\omega_{10} -d_{i1}\leq \kappa^3(q_{ij1}^+),\forall i,j (y_i=y_j), \label{jspf}\tag{${\rm JSP}$}\\
 & \omega^t_1 x_i + \omega_{10}-d_{i1}  \leq \kappa^3(q_{ij1}^-), \forall i,j (y_i=y_j), \nonumber\\
 & -d_{i1} \leq 0, \forall i, \nonumber\\
& -e_{i1}\leq 0,  \forall i,\nonumber\\
&  \mathbf{\omega}_1 \in \R^d, \omega_{10} \in \R,\nonumber\\
&\dsum_{j: u^+_{ijr}=0} \mu_{ijr}^3 + \dsum_{j: u^-_{ijr}=0} \mu_{ijr}^4  + \mu_{ir}^8 \geq C_1, \forall i,\nonumber\\
&\dsum_{j: q^+_{ijr}=0} \mu_{ijr}^5+ \dsum_{j: q^-_{ijr}=0} \mu_{ijr}^6 \leq C_2, \forall i, r,\nonumber\\
& \mu_{ir}^1 =0, \forall i, r \mbox{ with } \bar t_{ir}=0,\nonumber\\
& \mu_{ir}^2 =0, \forall i, r \mbox{ with } \bar t_{ir}=1,\nonumber\\
&\alpha_{ir} =  \mu_{ir}^1- \mu_{ir}^2 +\dsum_{i,j: u^+_{ij1}=0} \mu_{ijr}^3- \dsum_{i,j: u^-_{ij1}=0} \mu_{ijr}^4+\dsum_{i,j: q^+_{ij1}=0} \mu_{ijr}^5\\
& - \dsum_{i,j: q^-_{ij1}=0}  \mu_{ijr}^6,\forall i, \nonumber\\
& \mathbf{\mu} \geq 0.\nonumber
\end{align}
\noindent where $\mu_r\ge 0$ are the dual multipliers of the constraints in Problem \eqref{spr}.
\end{lem}

\noindent \textit{Proof.}
In order to prove the result, in what follows we derive the optimality conditions of the Lagrangean dual of \eqref{spr} for $r=2,\ldots,m$. Relaxing the constraints of \eqref{spr} with dual multipliers $\mu_r\ge 0$, the Lagrangean function for given values of the binary variables $\bar t$, $\bar \xi$, $\bar h$ and $\bar z$ (and consequently for values $\bar u^+$, $\bar u^-$, $\bar q^+$ and $\bar q^-$) is:

\begin{eqnarray*}
\begin{split}
\mathcal{L}_r(\omega_r, \omega_{0r}, d, e; \mathbf{\mu}) &=   C_1 \dsum_{i=1}^n  e_{ir}+  C_2\dsum_{i=1}^n  d_{ir}  + \mu_r^0 \Big(\frac{1}{2} \|\omega_r\|^2-\frac{1}{2}\|\omega_1\|^2\Big)\\
&+\dsum_{i} \mu^1_{ir} \Big(-\omega_{r}^t x_i - \omega_{r0}- \kappa^1(t_{ir})\Big)+\dsum_{i} \mu^2_{ir} \Big(\omega_{r}^t x_i + \omega_{r0}- \kappa^2(t_{ir})\Big)\\
&+ \dsum_{i,j}\mu^3_{ijr} \Big(-\omega_{r}^t x_i - \omega_{r0} - e_{ir} - \kappa^3(\bar u^+_{ijr})\Big)\\
&+ \dsum_{i,j} \mu^4_{ijr} \Big( \omega_{r}^t x_i +\omega_{r0} -e_{ir}  - \kappa^3(\bar u^-_{ijr})\Big)\\
&+ \dsum_{i, j} \mu^5_{ijr} \Big(-d_{ir} -\omega^t_r x_i -\omega_{r0}- \kappa^3(\bar q_{ijr}^+)\Big)\\
& +\dsum_{i, j} \mu^6_{ijr} \Big(-d_{ir} + \omega^t_r x_i + \omega_{r0}- \kappa^4(\bar q_{ijr}^-)\Big)\\
& + \dsum_{i} \mu^7_{ir}(-d_{ir}) +  \dsum_{i} \mu^8_{ir}(-e_{ir}). 
\end{split}
\end{eqnarray*}

Therefore, the KKT optimality conditions read as:
\begin{itemize}
\item $\mu^0_r \omega_r= \dsum_{i=1}^n \Big (\mu_{ir}^1- \mu_{ir}^2 + \dsum_{j: y_i=y_j} (\mu_{ijr}^3- \mu_{ijr}^4 +\mu_{ijr}^5 - \mu_{ijr}^6 )\Big) x_i$.
\item $\dsum_{i=1}^n \Big( \mu_{ir}^1- \mu_{ir}^2 +\dsum_{j: y_i=y_j} (\mu_{ijr}^3- \mu_{ijr}^4 +\mu_{ijr}^5 - \mu_{ijr}^6 )\Big) =0$.
\item $\dsum_{j: y_i=y_j}(\mu_{ijr}^3 +  \mu_{ijr}^4) + \mu^8_{ir}  = C_1$, for all $i$.
\item $ \dsum_{j: y_i=y_j} (\mu_{ijr}^5 + \mu_{ijr}^6 ) + \mu^7_{ir} = C_2$, for all $i$.
\item $\mu_r \geq 0$.
\end{itemize}

First of all, we observe that if $\| \omega_r \| < \| \omega_1\|$ at optimality then $\mu_r^0=0$ and actually, we do not have to consider the corresponding constraint nor the addend
$\frac{\mu_r^0}{2}(\| \omega_r \|^2 - \| \omega_1 \|^2)$ in the Lagrangean function. Hence, denoting by $\alpha_{ir} = \mu_{ir}^1- \mu_{ir}^2 +\dsum_{j: y_i=y_j} (\mu_{ijr}^3- \mu_{ijr}^4+\mu_{ijr}^5 - \mu_{ijr}^6)$, for all $i=1, \ldots, n$, and assuming that $\mu^0_r \neq 0$, the dual of \eqref{spr} reads as:

\begin{align*}
 \max_{\omega_r, \omega_{r0}, d, e; \mathbf{\mu}} & -\dfrac{\mu^0_r}{2} \| \omega_1\|^2 + \dfrac{1}{2\mu^0_r} \dsum_{i,j=1}^n \alpha_{ir}\alpha_{jr} x_i^t x_j -\dsum_{i} \mu_{ir}^1 \kappa^1(\bar t_{ir}) -\dsum_{i} \mu_{ir}^2 \kappa^2(\bar t_{ir}) \\
 & - \dsum_{i} \mu^3_{ir} \kappa^3(\bar u_{ijr}^+) - \dsum_{i} \mu^4_{ir} \kappa^4(\bar u_{ijr}^-) \\
&-  \dsum_{i,j} \mu^5_{ijr} \kappa^5(\bar q_{ijr}^+)-  \dsum_{i,j} \mu^6_{ijr} \kappa^6(\bar q_{ijr}^-)\\
\mbox{s.t. }& \dsum_{i=1}^n \alpha_{ir} x_i = \mu^0_r \omega_r,\\
&\dsum_{i=1}^n \alpha_{ir} =0, \\
& \dsum_{j: y_i=y_j} (\mu_{ijr}^3 + \mu_{ijr}^4) \leq C_1, \forall i,\\
&\dsum_{j: y_i=y_j} (\mu_{ijr}^5 + \mu_{ijr}^6 ) \leq C_2, \forall i,\\
&\alpha_{ir} =  \mu_{ir}^1- \mu_{ir}^2 +\dsum_{i,j: u^+_{ij1}=0} \mu_{ijr}^3- \dsum_{i,j: u^-_{ij1}=0} \mu_{ijr}^4+\dsum_{i,j: q^+_{ij1}=0} \mu_{ijr}^5 \\
& - \dsum_{i,j: q^-_{ij1}=0}  \mu_{ijr}^6,\forall i,\\
& \mathbf{\mu_r} \geq 0.
\end{align*}

Let us simplify further the expressions above. We observe that:

$$
\dsum_{i} \mu^1_{ir} \kappa^1(\bar t_{ir})+\dsum_{i} \mu^2_{ir} \kappa^2(\bar t_{ir}) = 0,
$$
$$
-\dsum_{i,j} \mu^3_{ijr} \kappa^3(\bar u^+_{ijr}) - \dsum_{i,j} \mu^4_{ijr} \kappa^3(\bar u^-_{ijr}) = \dsum_{i,j:\atop \bar u^+_{ijr}=0} \mu^3_{ijr} +  \dsum_{i,j:\atop \bar u_{ijr}^-=0} \mu^4_{ijr},
$$
and
\begin{eqnarray*}
\begin{split}
-  \dsum_{i,j} \mu^5_{ijr} \kappa^3(\bar q_{ijr}^+) &-  \dsum_{i,j} \mu^6_{ijr} \kappa^4(\bar q_{ijr}^-) = \dsum_{i,j:\atop \bar q^+_{ijr}=0} \mu^5_{ijr} +   \dsum_{i,j:\atop  \bar q^-_{ijr}=0} \mu^6_{ijr}.
\end{split}
\end{eqnarray*}
Using those equations and substituting the problem becomes:

\begin{align}
 \max_{\omega_r, \omega_0, d, e; \mathbf{\mu}} & -\frac{\mu^0_r}{2} \|\omega_1 \|^2 + \dfrac{1}{2\mu^0_r} \dsum_{i,j=1}^n \alpha_{ir}\alpha_{jr} x_i^t x_j +\dsum_{i,j: u^+_{ijr}=0} \mu_{ijr}^3 + \dsum_{i,j: u^-_{ijr}=0} \mu_{ijr}^4\nonumber\\
 & +\dsum_{i,j: q^+_{ijr}=0} \mu_{ijr}^5+ \dsum_{i,j: q^-_{ijr}=0} \mu_{ijr}^6\nonumber\\
\mbox{s.t. }& \dsum_{i=1}^n \alpha_{ir} x_i = \mu^0_r \omega_r, \nonumber\\
&\dsum_{i=1}^n \alpha_{ir} =0, \forall r,\nonumber\\
&\dsum_{j: u^+_{ijr}=0} \mu_{ijr}^3 + \dsum_{j: u^-_{ijr}=0} \mu_{ijr}^4  \leq C_1, \forall i,\nonumber\\
&\dsum_{j: q^+_{ijr}=0} \mu_{ijr}^5+ \dsum_{j: q^-_{ijr}=0} \mu_{ijr}^6 \leq C_2, \forall i, r,\label{dspr}\tag{${\rm DSP}_r$}\\
& \mu_{ir}^1 =0, \forall i, r \mbox{ with } \bar t_{ir}=0,\nonumber\\
& \mu_{ir}^2 =0, \forall i, r \mbox{ with } \bar t_{ir}=1,\nonumber\\
&\alpha_{ir} =  \mu_{ir}^1- \mu_{ir}^2 +\dsum_{i,j: u^+_{ij1}=0} \mu_{ijr}^3- \dsum_{i,j: u^-_{ij1}=0} \mu_{ijr}^4+\dsum_{i,j: q^+_{ij1}=0} \mu_{ijr}^5\\
& - \dsum_{i,j: q^-_{ij1}=0}  \mu_{ijr}^6,\forall i,\nonumber\\
& \mathbf{\mu} \geq 0.\nonumber
\end{align}

Using the strong duality in all the subproblems \eqref{spr} for $r=2,\ldots,m$, we can obtain the following expansion of the join subproblem \eqref{spr} that allows one the evaluation of $\Phi$ defined in problem \eqref{evalfi}.

\begin{align}
\hat\Phi( h, z, t,\xi) = & \min \left\{ \frac{1}{2} \|\omega_1\|^2+ C_1 \sum_{i=1}^n e_{i1}+C_2\sum_{i=1}^n d_{i1})  \right. \nonumber \\
 & + \max_{\omega_r, \omega_0, d, e; \mathbf{\mu}} \sum_{r=2}^m \left(-\frac{\mu^0_r}{2} \|\omega_1 \|^2 + \dfrac{1}{2\mu^0_r} \dsum_{i,j=1}^n \alpha_{ir}\alpha_{jr} x_i^t x_j
 +\dsum_{i,j: u^+_{ijr}=0} \mu_{ijr}^3\right. \\
 &\left.+ \dsum_{i,j: u^-_{ijr}=0} \mu_{ijr}^4 +\dsum_{i,j: q^+_{ijr}=0} \mu_{ijr}^5+ \dsum_{i,j: q^-_{ijr}=0} \mu_{ijr}^6\right)\nonumber\\
& \mbox{s.t. }  -\omega_{1}^tx_i - \omega_{10} \leq \kappa^1(t_{i1}), \forall i,\nonumber\\
& \omega_{1}^tx_i + \omega_{10}  \leq  \kappa^2(t_{i1}), \forall i, \nonumber\\
&-\omega_{1}^t x_i - \omega_{10} - e_{i1} \leq  \kappa^3(u_{ij1}^+), \forall i,j,\nonumber\\
 &\omega_{1}^t x_i +\omega_{10} -e_{i1} \leq  \kappa^3(u_{ij1}^-), \forall i,j, r, \nonumber\\
 & -\omega^t_1 x_i -\omega_{10} -d_{i1}\leq \kappa^3(q_{ij1}^+),\forall i,j (y_i=y_j), \nonumber\\
 & \omega^t_1 x_i + \omega_{10}-d_{i1}  \leq \kappa^3(q_{ij1}^-), \forall i,j (y_i=y_j), \nonumber\\
& -d_{i1} \leq 0, \forall i, \nonumber\\
& -e_{i1}\leq 0, \forall i,\nonumber\\
&  \mathbf{\omega}_1 \in \R^d, \omega_{10} \in \R,\nonumber\\
&\dsum_{j: u^+_{ijr}=0} \mu_{ijr}^3 + \dsum_{j: u^-_{ijr}=0} \mu_{ijr}^4  \leq C_1, \forall i,\nonumber\\
&\dsum_{j: q^+_{ijr}=0} \mu_{ijr}^5+ \dsum_{j: q^-_{ijr}=0} \mu_{ijr}^6 \leq C_2, \forall i, r,\nonumber\\
& \mu_{ir}^1 =0, \forall i, r \mbox{ with } \bar t_{ir}=0,\nonumber\\
& \mu_{ir}^2 =0, \forall i, r \mbox{ with } \bar t_{ir}=1,\nonumber\\
&\alpha_{ir} =  \mu_{ir}^1- \mu_{ir}^2 +\dsum_{i,j: u^+_{ij1}=0} \mu_{ijr}^3- \dsum_{i,j: u^-_{ij1}=0} \mu_{ijr}^4+\dsum_{i,j: q^+_{ij1}=0} \mu_{ijr}^5 \\
&- \dsum_{i,j: q^-_{ij1}=0}  \mu_{ijr}^6,\forall i,\nonumber\\
& \mathbf{\mu} \geq 0.\nonumber
\end{align}
The usual transformation of the maximum in the objective function gives rise to the equivalent reformulation of the above problem as (\ref{jspf}).
\hfill $\square $

\end{document}